\documentclass[11pt]{amsart}
\usepackage[utf8]{inputenc}
\usepackage{lmodern}
\usepackage{bold-extra}
\usepackage[all]{xy}
\usepackage{dsfont}
\usepackage{hyperref}
\usepackage{color}
\usepackage{esint}
\usepackage{mathtools}
\usepackage[T1]{fontenc}
\usepackage{textcomp}
\mathtoolsset{showonlyrefs}
\usepackage{slashed}
\usepackage{bigints}
\usepackage{psfrag}
\usepackage{tikz-cd}

\usepackage{scalerel}
\usepackage{amsmath,amssymb,amsthm,amsfonts,mathrsfs,bm}

\usepackage{exscale,relsize}
\usepackage{textgreek}
\usepackage{epsfig,graphics}

\usepackage[margin=0.9in]{geometry}

\newcommand{\Abracket}[1]{\left<#1\right>} 
\newcommand{\parenthesis}[1]{\left(#1\right)} 
\newcommand{\braces}[1]{\left\{#1\right\}} 

\newcommand{\R}{\mathbb{R}}

\newcommand{\dd}{\mathop{}\!\mathrm{d}}
\newcommand{\eps}{\varepsilon}

\newcommand{\ii}{\infty}
\newcommand{\graf}[1]{\left\{\begin{array}{ll}#1\end{array}\right.}

\newcommand{\al}{\alpha}

\newcommand{\sg}{\sigma}
\newcommand{\ga}{\gamma}
\newcommand{\om}{\Omega}
\newcommand{\lm}{\lambda}

\newcommand{\pa}{\partial}

\newcommand{\fbi}{{\textbf{(}\mathbf F\textbf{)}_{\mathbf I}}}
\newcommand{\prl}{{\textbf{(}\mathbf P\textbf{)}_{\mathbf \lm}}}

\newcommand{\pqm}{{\textbf{(}\mathbf Q\textbf{)}_{\mathbf \mu}}}

\newcommand{\vl}{\eta_{\sscp \lm}}
\newcommand{\vxi}{\xi}

\newcommand{\sscp}{\scriptscriptstyle}
\newcommand{\dsp}{\displaystyle}
\newcommand{\noi}{\noindent}

\newcommand{\ssl}{\sscp \lm}
\newcommand{\ml}{m_{\sscp \lm}}
\newcommand{\vla}{v^{({\sscp \lm})}}
\newcommand{\gall}{\ga_{\ssl}}
\newcommand{\all}{\al_{\ssl}}

\newcommand{\el}{E_{\ssl}}

\newcommand{\tl}{\tau_{\ssl}}

\newcommand{\pl}{\psi_{\sscp \lm}}
\newcommand{\xil}{\ul}
\newcommand{\ul}{u_{\sscp \lm}}
\newcommand{\wl}{w_{\sscp \lm}}

\newcommand{\rife}[1]{(\ref{#1})}
\newcommand{\ov}[1]{\overline{#1}}
\newcommand{\un}[1]{\underline{#1}}

\newcommand{\mul}{\mu_{\ssl}}

\newcommand{\rl}{\mbox{\Large \textrho}_{\!\sscp \lm}}

\newcommand{\rlq}{(\mbox{\Large \textrho}_{\! \sscp \lm})^{\frac1q}}

\newcommand{\ino}{\int_{\Omega}}

\DeclareMathOperator{\dx}{\dd {x}}

\DeclareMathOperator{\Eigen}{Eigen}

\DeclareMathOperator{\Ker}{Ker}

\DeclareMathOperator{\Span}{Span}

\DeclareMathOperator{\Spect}{Spect}

\newtheorem{thm}{Theorem}[section]

\newtheorem{lemma}[thm]{Lemma}
\newtheorem{prop}[thm]{Proposition}
\newtheorem{cor}[thm]{Corollary}

\newtheorem{rmk}[thm]{Remark}

\title[Constrained-stable solutions]{The Rabinowitz continuum of subcritical Gelfand problems \\ and free boundary-type equations arising in plasma physics}


\author[D. Bartolucci]{Daniele Bartolucci}
\address{Daniele Bartolucci, Department of Mathematics, University of Rome \emph{"Tor Vergata"}, Via della ricerca scientifica n.1, 00133 Roma.}
\email{bartoluc@mat.uniroma2.it}

\author[A. Jevnikar]{Aleks Jevnikar}
\address{Aleks Jevnikar, Department of Mathematics, Computer Science and Physics, University of Udine, Via delle Scienze 206, 33100 Udine, Italy.}
\email{aleks.jevnikar@uniud.it}

\author[J. Wei]{Juncheng Wei}
\address{Juncheng Wei, Department of Mathematics, Chinese University of Hong Kong, Shatin, Hong Kong.}
\email{wei@math.cuhk.edu.hk}

\author[R. Wu]{Ruijun Wu}
\address{Ruijun Wu, School of mathematics and statistics, Beijing Institute of Technology, Zhongguancun South Street No. 5, 100081 Beijing, P.R.China.}
\email{ruijun.wu@bit.edu.cn}

\thanks{ D.B.
is partially supported by the MIUR Excellence Department Project MatMod@TOV
awarded to the Department of Mathematics, University of Rome ``Tor Vergata'' and by PRIN project 2022, ERC PE1\_11,
``{\em Variational and Analytical aspects of Geometric PDEs}'', by INdAM-GNAMPA project ``{\em Analisi qualitativa di problemi differenziali non lineari}'' and by the E.P.G.P. Project sponsored by the University of Rome ``Tor Vergata''.\\
 \indent A.J. is partially supported by INdAM-GNAMPA project ``{\em Analisi qualitativa di problemi differenziali non lineari}'' and PRIN Project 20227HX33Z ``{\em Pattern formation in nonlinear phenomena}''.\\
\indent D.B. and A.J. are member of the INDAM Research Group ``Gruppo Nazionale per l'Analisi Matematica, la Probabilità e le loro Applicazioni''.\\
\indent J.W. is partially supported by GRF fund of RGC of Hong Kong entitled
``\emph{New frontiers in singularity formations of nonlinear partial differential equations}''.
}

\thanks{We would like to express our warmest thanks to P. Korman for very fruitful exchange of ideas about global bifurcation of Gelfand problems.}
\begin{document}

\begin{abstract}
    The qualitative behavior of the Rabinowitz unbounded continuum of
subcritical Gelfand problems is well known on balls in any dimension.
We don't know of any such sharp and detailed description otherwise,
which is our motivation to look for a new approach to the problem. The
underlying idea is to describe solutions of Gelfand problems via
suitably defined constrained problems of free boundary-type arising in
plasma physics and to replace the usual $L^\infty$ norm of the solution
with the energy of the plasma. Toward this goal, we first solve a long
standing open problem of independent interest about the uniqueness of
solutions of Grad-Shafranov type equations. Thus, we exploit these
unique solutions to detect a curve containing both minimal and non
minimal solutions of the associated Gelfand problem. In other words we
come up with a new global parametrization of the Rabinowitz continuum,
the monotonicity of the energy along the branch providing a meaningful
generalization of the classical pointwise monotonicity property of
minimal solutions, suitable to describe non minimal solutions as well.
On a ball in any dimension, we come up as expected with a bell-shaped
profile of the full branch of solutions of the Gelfand problem.
\end{abstract}

\maketitle

\section{Introduction}
\noi Let~$\Omega\subset \R^N$,~$N\ge 2$ be a bounded domain of class~$C^{3}$ and set
$$
p_{_N}=\graf{+\ii,\;N=2 \\ \frac{N}{N-2},\; N\geq 3.}
$$
For $1<p<p_{_N}$ we wish to develop a general method to describe the qualitative behavior of the Rabinowitz unbounded continuum (\cite{Rabinowitz1971some})
of classical solutions of the Gelfand problem
$$
\graf{-\Delta v =\mu (1+v)^p\quad \mbox{in}\;\;\om\\
v>0 \quad \mbox{in}\;\;\om\\
v=0 \quad \mbox{on}\;\;\pa\om
}\qquad \pqm
$$
It is well known (see e.g. \cite{CrandallRabinowitz1975continuation}) that there exists $\mu_*=\mu_*(\om,p)>0$ such that~$\pqm$ has at least two solutions for~$\mu<\mu_*$ (one of which is a minimal solution, the other one a mountain-pass solution), exactly one solution for $\mu=\mu_*$ and no solution at all for $\mu>\mu_*$, the well known branch of minimal solutions being smooth and monotone as a function of $\mu$.
However results of this sort are concerned with the counting of solutions and not with the connectivity and/or with the qualitative behavior of the set of non-minimal solutions.
Besides the general result in \cite{Rabinowitz1971some}, the existence of continua of solutions of~$\pqm$ containing both minimal and non-minimal solutions is well-known for any subcritical~$p<\frac{N+2}{N-2}$ and convex domains, see \cite{Lions1982existence}.
The claim in that case is about the existence of a continua of solutions containing minimal solutions for any $\mu\leq \mu_*$ and non-minimal solutions characterized by the fact that they diverge as $\mu\to 0^+$, but with no reference in general to the monotonicity of~$\|v_\mu\|_{\ii}$ along the branch.

\noi If $\om\subset \R^N$ is a ball it is well-known that solutions are radial (\cite{GNN1979symmetry}) and that~$\pqm$ has exactly two solutions for $\mu<\mu_*$ for any $p\leq \frac{N+2}{N-2}$ (\cite{JosephLundgren1972quasilinear, Korman2012global}).
In particular all solutions are contained in a unique curve originating from $(\mu,v)=(0,0)$ which makes only one turn to the left and then diverges as $\mu\to 0^+$, see \cite{Korman2012global}.
Since $v_{\mu}(0)$ is monotonic increasing along the branch, this shows that the global branch has, in the plane $(\mu,v_{\mu}(0))$, the classical bell-shaped profile, see for example \cite{Lions1982existence,OuyangShi1999exact}.
In dimension $N=2$, a general result in \cite{Holzman1994uniqueness} states that if $\om$ is symmetric and convex with respect to coordinate directions, then all the solutions lie on a curve of the form $t\mapsto (\mu(t),v_{\mu(t)}(0))$, where $v_{\mu(t)}(0)$ is strictly increasing in $t$.
However, with the unique exception of the above mentioned results on balls, it seems that the monotonicity of $\mu(t)$ is not known so far, neither in dimension $N=2$ for $\om$ symmetric and convex with respect to coordinate directions.
See \cite{BJ2022uniqueness} for a rather general result about the qualitative behavior of the Rabinowitz
unbounded continuum, including the global behavior of $\mu(t)$, in case $N=2$ for the exponential nonlinearity.

\

 These are our motivations to try to develop a new general method for the description of the qualitative behavior
of the Rabinowitz unbounded continuum of classical, minimal and non minimal solutions of $\pqm$ with $p<p_{_N}$, extending somehow the classical results about minimal solutions in \cite{CrandallRabinowitz1975continuation}.
Toward this goal we pursue the constrained problem,
$$
\graf{-\Delta \psi =[\al+{\lm}\psi]_+^p\quad \mbox{in}\;\;\om\\ \\
\bigintss\limits_{\om} { [\al+{\lm}\psi]_+^p}=1\\ \\
\psi=0 \quad \mbox{on}\;\;\pa\om
}\qquad \prl
$$
for $p<p_{_N}$, $\lm\geq 0$ and the unknowns $\al\in\R$ and $\psi \in C^{2,r}_{0}(\ov{\om}\,)$.
If $(\al_\lm,\psi_\lm)$ solves $\prl$ for some~$\lm>0$ and $\al_\lm>0$, then $v^{(\lm)}=\frac{\lm}{\al_\lm}\psi_\lm$ solves $\pqm$ for $\mu=\mu_\lm=\lm\al_{\lm}^{p-1}$. Assuming for simplicity $|\om|=1$, the underlying idea of the connection with $\pqm$ is thus to gather enough information to follow
“the” curve of solutions of $\prl$ from $\lm=0$ and $\all=\al_0=1$, up to “the first value” of $\lm$ where $\al_\lm$
approaches $\al=0$ from above, while keeping in particular the monotonicity of a suitably defined energy.

\bigskip

Let $q$ be the H\"older conjugate of~$p$, i.e.~$\frac{1}{p}+\frac{1}{q}=1$.
For a fixed $(\al_\lm,\psi_\lm)$ which solves $\prl$, we will adopt the following notations,
$$
 \rl= [\al_\lm+\lm \psi_\lm]_+^p \quad\mbox{\rm and }\quad {\rlq} =[\al_\lm+\lm \psi_\lm]_+^{p-1}.
$$
Problem $\prl$ is of independent interest due to its relevance to Tokamak's plasma physics (\cite{Freidberg2014IdealMHD, Kadomtsev1996nonlinear, Stacey2012fusion}),
which motivated the lot of work done to understand existence, uniqueness, multiplicity of solutions and existence/non-existence/structure of the free boundary $\pa\{x\in\om\,:\,\al+\lm \psi>0\}$ of $\prl$,
see e.g.~\cite{BandleMarcus1982boundary,BandleSperb1983qualitative,BJ2022uniqueness,BJW2024sharp, Berestycki1980free, CaffarelliFriedman1980asymptotic, CaoPengYan2010multiplicity,FriedmanLiu1995free,LiPeng2015multipeak,Liu2014multiple,SuzukiTakahashi2016critical,Wei2001multiple,Wolansky1996critical} and references quoted therein. In fact $\prl$ is equivalent to the well-known Grad-Shafranov type equation, see~$\fbi$ in section \ref{sect:calc in ball}.
In particular due to the fundamental results in \cite{Berestycki1980free}, it is well known that for any~$\lambda>0$ there exists at least one solution of~$\prl$.
The variational formulation adopted in \cite{Berestycki1980free} is closely related to minimal free energy principles with non-extensive entropy, where $-\lm$ plays the role of the inverse statistical temperature, see Appendix A in \cite{BJW2024sharp} and \cite{CLMP1995special} for further details.

\noi On the other side, due to the constraint in $\prl$, we miss the classical notion of minimal solutions (\cite{CrandallRabinowitz1975continuation}) for this problem.
Remark that this is not just a technical point, as we shortly discuss hereafter.
First of all it has been recently shown in \cite{BJ2022uniqueness,BJW2024sharp} that in a neat interval of values of $\lm$ there exists a unique positive (i.e. $\all>0$) solution of $\prl$ forming a branch $\mathcal{G}_0(\om)$
along which $\all$ is decreasing while a suitably defined energy is increasing, see Theorems A and B below. Actually, still due to the constraint, it is not true in general that any first eigenfunction of the associated linearized operator (see \rife{lineq0.1} below) must be simple and neither that if the first eigenvalue is positive then the maximum principle holds, see \cite{BJ2021global} for an example illustrating this point.
As a consequence, unlike classical branches of Gelfand problems emanating from $(\mu,v)=(0,0)$ (\cite{CrandallRabinowitz1975continuation}), it is not at all clear that $\rl=[\all+\lm \pl]_+^p$ is pointwise
monotonic increasing along $\mathcal{G}_0(\om)$.
A workaround to this problem has been recently found in \cite{BJ2022new,BJ2022uniqueness,BJW2024sharp}, yielding suitable monotonicity property through the bifurcation analysis of $\prl$ for solutions with $\all>0$, again referring to Theorems A and B below.
However, in particular for $N\geq 3$, the requirement of the positivity of $\all$ causes a major obstacle in the understanding of the uniqueness and continuation of branches of solutions.
Motivated by the analysis of $\pqm$, we solve this problem here by introducing a weighted first eigenvalue in the same spirit of \cite{ManesMicheletti1973estensione}, where the weight, which is $\rlq=[\all+\lm \pl]_+^{p-1}$, is allowed to vanish in a “large” set.
This in turn allows the description of a branch of stable (i.e. with positive weighted first eigenvalue) solutions for $\prl$, see Theorem \ref{thm1} below.

\bigskip

We follow the traditional convention to assume that the domain has unit volume, $|\Omega|=1$.
In particular $\mathbb{D}_{N}$ denotes a ball in $\R^N$ centered at the origin of unit volume $|\mathbb{D}_N|=1$, whose radius is denoted by~$R_N$.
For fixed $t\geq 1$ we define
$$
\Lambda(\om,t)=\inf\limits_{w\in H^1_0(\om), w\equiv \!\!\!\!/ \;0}
\dfrac{\ino |\nabla w|^2}{\left(\ino |w|^{t}\right)^{\frac2t}}\,,
$$
which is related to the best constant in the Sobolev embedding $\|w\|_p\leq \mathcal{C}_S(\om,p)\|\nabla w\|_2$ via
\begin{align}\nonumber
 \mathcal{C}_S(\om,p)=\Lambda^{-1/2}(\om,p) \quad \mbox{ for } p\in[1,2p_{_N}).
\end{align}

\noi {\bf Definition.} {\it We say that a solution $(\all,\pl)$ of {\rm $\prl$} is {\bf positive}
{\rm[}resp. {\bf non-negative}{\rm]} if
$\all> 0$ {\rm[}resp. $\all\geq 0${\rm]}}.\\
\noi It has been recently proved in \cite{BJW2024sharp} that the following quantity is well-defined and strictly positive:
$$
\lm_+(\om,p)=
\sup\left\{\delta>0\,:\, \all>0 \mbox{ for any solution } (\all,\pl)
\mbox{ of } \prl \mbox{ with }\lm<\delta\right\}.
$$

\noi {\bf Remark.} {\it We will use the fact that $\lm_+(\om,p)$ is always finite, as readily follows
by the variational formulation of solutions of {\rm $\prl$}. See also{\rm ~\cite[Appendix A]{BJW2024sharp}}.}\\

\noi {\bf Definition.} {\it Let $(\all,\pl)$ be a solution of {\rm $\prl$}.
The {\bf energy} of $(\all,\pl)$ is
$$
\el=\frac12\ino |\nabla \pl|^2=\frac12 \ino  \mbox{\rm$\rl$}\pl.
$$
}

\noi Let us define
$$
\lm_0(\om,p):=\frac{1}{p}\Lambda(\om,2p),\quad  \om\subset \R^N,\; N\geq 2,
$$
and
$$
\lm_1(\om,p):=\left(\frac{8\pi}{p+1}\right)^{\frac{p-1}{2p}}\Lambda^{\frac{p+1}{2p}}(\om,p+1),\quad \om\subset \R^2.
$$

\bigskip

\noi The following results about positive solutions have been recently proved in~\cite{BJ2022uniqueness,BJW2024sharp}.\\
{\bf Theorem A}{(\cite{BJ2022uniqueness})}. {\it Let $N\geq 2$ and $p\in [1,p_N)$.
For any $\lm<\lm_0(\om,p)$ there exists a unique positive solution to {\rm $\prl$},
defining a real analytic curve $\mathcal{G}_0(\om)$,
such that $\al_0=1$, $2 E_0=2 E_0(\om)$ is the torsional rigidity of $\om$ and
$$
\frac{d \all}{d\lm}<0,\quad \frac{d \el}{d\lm}>0,\quad \forall (\all,\pl)\in\mathcal{G}_0(\om).
$$
In particular,
$$
 \lim\limits_{\underset{\lm\to \lm_0(\om,p)^{-}}{(\all,\pl)\in\mathcal{G}_0(\om)}}\all=0
$$
if and only if $p=1$.}

\medskip

\noi The situation is better for $N=2$, where we have a sharp "positivity" estimate and, as a consequence, an unconditional uniqueness result.\\
\noi{\bf Theorem B}{(\cite{BJ2022uniqueness, BJW2024sharp})}. {\it Let $N=2$ and $p\in [1,+\ii)$, then $\lm_+(\om,p)\geq \lm_1(\om,p)$
where the equality holds if and only if either $p=1$ or $\om=\mathbb{D}_2$.
Moreover, for any $\lm<\min\{\lm_0(\om,p),\lm_1(\om,p)\}$ there exists a unique solution to {\rm $\prl$},
defining a
real analytic curve of positive solutions which we denote by~$\mathcal{G}_0(\om)$,
such that $\al_0=1$, $2 E_0=2 E_0(\om)$ is the torsional rigidity of $\om$ and
$$
\frac{d \all}{d\lm}<0,\quad \frac{d \el}{d\lm}>0,\quad \forall (\all,\pl)\in\mathcal{G}_0(\om).
$$}

\noi We tend to believe that $\lm_1(\om,p)>\lm_0(\om,p)$, which is always true at least for $p$ large enough and for $\om=\mathbb{D}_2$, see \cite{BJW2024sharp}.
It is not clear whether or not the sharp positivity estimates about $\lm_+(\om,p)$ obtained in Theorem B can be extended to higher dimensions $N\geq 3$.
For later purposes we also recall the following result in \cite{BJW2024sharp}. Let~$u_0$ be the unique (\cite{GNN1979symmetry}) solution of the Emden equation
\begin{align}\label{Emden0}
    \graf{-\Delta u_0 =u_0^p\quad \mbox{in}\;\;B_1\\
u_0>0 \quad \mbox{in}\;\;B_1\\
u_0=0 \quad \mbox{on}\;\;\pa B_1,
}
\end{align}
and set $$
I_p=\int\limits_{B_{1}} {\dsp {u}_0^p}\dx.
$$
Note that $R_N$ denotes the radius of $\mathbb{D}_{N}$ and let us set,
\begin{align}\label{lmpN}
 \lm_+(\mathbb{D}_{N},p)={I_p^{1-\frac{1}{p}}}{R_N^{-\frac{N}{p}(1-\frac{p}{p_{_N}})}}.
\end{align}

\noi
Then we have,\\

\noi{\bf Theorem C}{(\cite{BJW2024sharp})}. {\it Let $\om=\mathbb{D}_{N}\subset \R^N$, $N\geq 3$. Then, for any $p\in (1,p_{_N})$, {\rm $\prl$} admits a unique solution. Moreover
$\all>0$ if and only if $\lm<\lm_+(\mathbb{D}_N,p)$.}

\bigskip

\noi However, in particular for $N\geq 3$, it is still not clear whether or not the unique positive solutions in Theorem A are in fact the unique solutions of $\prl$, although this is indeed the case for~$\lambda>0$ sufficiently small (\cite{BJW2024sharp}).
We solve this problem here by refined spectral estimates for the linearized problem naturally associated with $\prl$.\\

First of all it can be shown (see section \ref{sec3} and in particular \eqref{lineq0.1}), that the "natural" non-local
eigenvalue problem associated to the linearized equation for $\prl$ evaluated at a fixed solution $(\all,\pl)$ takes the form
$$
-\Delta \phi-\tl \rlq [\phi]_{\ssl}=\sg\rlq [\phi]_{\ssl},
$$
where, here and in the sequel, $\sg\in \R$ denotes an eigenvalue relative to $\phi$,
$$
\tl=\lm p,
$$
and
$$
\Abracket{\phi}_{\ssl}\coloneqq \dfrac{\ino \rlq  \phi}{\ino \rlq},
\qquad
[\phi]_{\ssl}\coloneqq  \phi \,- \Abracket{\phi}_{\ssl},\qquad \forall\,\phi \in L^2(\om).
$$
We can prove that, for any fixed solution $(\all,\pl)$, the weighted first eigenvalue is well defined, which we will always denote by $\sg_1(\all,\pl)$, see Lemma \ref{lem:spectral}.
Moreover, let us define
$$
\lm_*(\om,p)\coloneqq
\sup\left\{\ov{\lm}>0\,:\, \sg_1(\all,\pl)>0 \mbox{ for any solution } (\all,\pl)
\mbox{ of } \prl \mbox{ with }\lm<\ov{\lm}\right\}.
$$

\bigskip

Our first result is a spectral estimate of independent interest, which implies in particular that the unique positive solutions of $\prl$ in Theorem A are in fact the unique solutions of $\prl$.
At least to our knowledge, this is the first unconditional uniqueness result about Grad-Shafranov type equations, solving a long standing open problem which dates back to the pioneering results in \cite{Berestycki1980free}, see~\cite{BJ2022uniqueness,BJW2024sharp} and references quoted therein.
\begin{thm}\label{thm1}
 Let $N\geq 2$ and $p\in [1,p_N)$, then $\lm_*(\om,p)>\lm_0(\om,p)$. Moreovr $\lm_+(\om,p)\geq \lm_0(\om,p)$ where the equality holds if and only if $p=1$, and
for any $\lm\in [0,\lm_*(\om,p))$ there exists a unique solution to {\rm $\prl$},
defining a $C^1$ curve of solutions, denoted by $\mathcal{G}_*(\om)$, such that $\al_0=1$, $2 E_0=2 E_0(\om)$ is the torsional rigidity of $\om$ and
$$
\frac{d \all}{d\lm}<0,\quad \frac{d \el}{d\lm}>0,\quad \forall (\all,\pl)\in\mathcal{G}_*(\om).
$$
\end{thm}

\noi Remark that the claim about the $C^1$ regularity of the map $\lm \mapsto (\all,\pl)$ refers to the $\R\times C^{2,r}_{0}(\ov{\om}\,)$ topology as far as
$p\in (1,p_N)$, while if $p=1$ it refers to the $\R\times W^{2,t}_{0}(\om)$ topology, for some $t>N$.\\  

\noi Along the branch $\mathcal{G}_*(\om)$ we have $\sg_1(\all,\pl)>0$ which is why
we say that $\mathcal{G}_*(\om)$ is a branch of \un{stable} solutions. In other words, a solution of $\prl$ is said to be
stable if the first eigenvalue of the linearized
equation is positive.  Remark also that, although we don't know much about the pointwise monotonicity of $\pl$, as far as
$\sg_1(\all,\pl)>0$  we have that
$\frac{d \el}{d\lm}>0$ and $\frac{d \all}{d\lm}<0$, see Propositions \ref{pr-enrg} and \ref{pr3.2.best} below.\\

\noi As a corollary of Theorem \ref{thm1} we come up with a new approach relevant to the understanding of the qualitative behavior of the branch of solutions of
$\pqm$ starting at $(\mu,v)=(0,0)$, also known as the Rabinowitz continuum (\cite{Rabinowitz1971some}), including non minimal solutions. Indeed we reduce the problem of the existence of a \un{global strictly monotone} (in terms of $\el$) parametrization to an estimate about $\lm_*(\om,p)$ and $\lm_+(\om,p)$.

\begin{thm}\label{thm:gel}
Let $N\geq 2$ and let $(\all,\pl)$ be the unique solutions of {\rm $\prl$} for
$$
\lm<\min\{\lm_*(\om,p),\lm_+(\om,p)\}.
$$
Then:
\begin{itemize}
    \item[(i)] The pair~$(\all,\pl)$, as well as~$\el$,  are real analytic in~$\lm$, with
        \begin{align}
            \frac{\dd \all}{\dd\lm}<0, & & \frac{\dd \el}{\dd \lm}>0, & & \forall \lm\in [0 , \; \min\braces{\lm_*(\om,p),\lm_+(\om,p)}  ).
        \end{align}
    \item[(ii)] The function~$\vla=\frac{\lm}{\all}\pl$ solves {\rm $\pqm$} for $\mu=\mul=\lm\all^{p-1}$, and the energy~$\el$ takes the form,
        \begin{align}\label{enrg}\el=\frac12\left(\frac{\all}{\lm}\right)^2\mul\ino (1+\vla)^p\vla.
        \end{align}
        Furthermore,
        \begin{align}
            \mathcal{R}_+=\{(\mul,\vla)\,:\,\lm\in [0,\min\{\lm_*(\om,p),\lm_+(\om,p)\})\}
        \end{align}
        is a real analytic branch of solutions of {\rm $\pqm$}.
    \item[(iii)] Assume that~$\lm_*(\om,p)> \lm_+(\om,p)$, then as $\lm\nearrow \lm_+(\om,p)$ we have,
         \begin{align}\label{iii-thm2}
            \al\searrow 0^+,\qquad \mul\searrow 0^+, \qquad  \|\vla\|_{\ii}\to +\ii,
        \end{align}
        and in particular $\mathcal{R}_+$ is unbounded.
\end{itemize}

\end{thm}

\noi As far as $\lm_*(\om,p)\geq \lm_+(\om,p)$, it is useful to denote
\begin{align}
    \overline{\mathcal{G}_{+}}(\om,p)\coloneqq  \mbox{  the curve of unique solutions of } \prl \mbox{ with } 0\leq \lm\leq \lm_+(\om,p)
\end{align}
and
\begin{align}
    {\mathcal{G}}_{+}(\om,p)\coloneqq  \mbox{ the portion of }  \overline{\mathcal{G}_{+}}(\om,p) \mbox{ with } 0\leq \lm< \lm_+(\om,p).
\end{align}
Due to Theorem \ref{thm1} and since $\lm_+(\om,p)$ is always finite, the proof of Theorem \ref{thm:gel} is deduced just by following
the curve of solutions of $\prl$ as far as either $\lm<\lm_*(\om,p)$ or else all along ${\mathcal{G}}_{+}(\om,p)$ up to $\lm_+(\om,p)$,
which, as far as $\lm_*(\om,p)>\lm_+(\om,p)$,  is “the first value” of $\lm$ where $\all$
approaches $\al=0$ from above, while keeping in particular the monotonicity of $\el$.
\begin{rmk}\label{rem:CS} {\it In particular, since $\el$ is the half of the mean value of $\pl$ with respect to the weight \mbox{\rm $\rl$}, we think that the monotonicity of $\el$ is a {good extension of the classical notion \mbox{\rm (\cite{CrandallRabinowitz1975continuation})} of monotonicity} {of minimal solutions},
which is of particular interest whenever solutions on $\mathcal{G}_*(\om)$ does not anymore correspond via
$\vla=\frac{\lm}{\all}\pl$ to minimal but in fact to \un{non-minimal solutions} of {\rm $\pqm$}.
We also observe that the quantity~$\frac{\vla}{\pl}$ is monotonically strictly increasing along~$\mathcal{G}_*$.
Actually the result shows that, as far as $\lm_*(\om,p)>\lm_+(\om,p)$, the global
parametrization $\mathcal{R}_+$ of solutions of {\rm $\pqm$}
built with the stable solutions of {\rm $\prl$} along ${\mathcal{G}}_{+}(\om,p)$ includes
a {full branch} of {non-minimal solutions} (since of course minimal solutions $v_{\mu}$ vanish {\rm (\cite{CrandallRabinowitz1975continuation})} in the limit $\mu\to 0^+$). However, as far as $\lm_*(\om,p)>\lm_+(\om,p)$, we also have the monotonicity of the solutions all along the branch, including \un{non-minimal solutions}. We are not aware of any criteria of this sort suitable to catch the monotonicity of the solutions behind the minimal branch.\\
Clearly, for those domains for which $\lm_*(\om,p)> \lm_+(\om,p)$, $\mathcal{R}_+$ would be just the Rabinowitz unbounded (positive) continuum {\rm (\cite{Rabinowitz1971some})}  relative to {\rm $\pqm$} and therefore would coincide, for $\om$ convex, with the continuum obtained in Theorem 2.2 in {\rm \cite{Lions1982existence}}. We recall that, as mentioned at the very beginning of the introduction, the monotonicity all along the branch of the
 $L^\ii$ norm of the solutions is well known on balls (see \cite{Korman2012global} and references quoted therein) and on planar domains symmetric and convex w.r.t. coordinate directions (\cite{Holzman1994uniqueness}).\\
The underlying idea of this new method is that a careful analysis of a subcritical, superlinear equation with one constraint, which is {\rm $\prl$}, is effective in describing the continuum of solutions of {\rm $\pqm$} with Morse index not larger than one.
Indeed, the continuum of solutions described in {\rm \cite{Lions1982existence}} contains a full branch of non-minimal solutions, which in turn can be
obtained by a careful use of the mountain pass theorem, see {\rm (\cite{CrandallRabinowitz1975continuation})}.
However, under suitable non-degeneracy assumptions, it is well-known that mountain pass solutions have Morse index not larger than one {\rm (\cite{Hofer1984note})}. This is why we refer to the branch $(\mul,\vla)$ as the
branch of {constrained-stable} solutions of {\rm $\pqm$}.}
\end{rmk}

\bigskip

Therefore, as a consequence of Theorem \ref{thm:gel}, we are naturally interested in finding sufficient conditions which guarantee $\lm_*(\om,p)> \lm_+(\om,p)$, which in turn would yield the existence of a monotone (in terms of $\el$) global parametrization of the Rabinowitz unbounded continuum $\mathcal{R}_+$.
An elementary estimate shows that $\sg_1(\all,\pl)>0$ as far as $\lm\leq \lm_0(\om,p)$ (see Proposition \ref{preig} below) which is unfortunately far from enough for our purposes.\\

\noi{\bf OPEN PROBLEM} Assume $N\geq 2$, $1<p<p_{_N}$ and $\om$ convex, is it true that $\lm_*(\om,p)> \lm_+(\om,p)$?\\

\noi The proof of $\lm_*(\om,p)> \lm_+(\om,p)$ seems to be a hard task in general. Interestingly enough
we succeed in proving the weak inequality for $\Omega=\mathbb{D}_{N}$. Indeed we have,
\begin{thm}\label{thm:lm*} Let~$p\in (1,p_{_N})$, then we have,
$$ \lm_*(\mathbb{D}_N,p) \geq \lm_+(\mathbb{D}_N,p),$$
and \eqref{iii-thm2} holds true, that is, as $\lm\nearrow \lm_+(\mathbb{D}_N,p)$, we have,
\begin{align}
            \al\searrow 0^+,\qquad \mul\searrow 0^+, \qquad  \|\vla\|_{\ii}\to +\ii.
        \end{align}
\end{thm}

\noi The proof is obtained by contradiction via new results of independent interest about the regularity properties of $\all$, to be used together with some
known facts, such as a bending  lemma of Crandall-Rabinowitz type and the spectral properties of variational solutions, see sections \ref{sect:calc in ball} and \ref{sec:lm*} for further details.\\

At this point, assuming that for some domain $\om$ we were able to prove that either
$\lm_*(\om,p)> \lm_+(\om,p)$ or $\lm_*(\om,p)\geq \lm_+(\om,p)$ and \eqref{iii-thm2} holds true, we would also like to describe the monotonicity of $\mu=\mul=\lm\all^{p-1}$ which, together with the monotonicity of $\el$, would yield a complete description of the qualitative behavior of the Rabinowitz continuum $\mathcal{R}_+$ in the plane $(\mu,E)$. However, remark that $\mul=\lm \all^{p-1}$ and that in this case $\all$ would be decreasing along ${\mathcal{G}}_{+}(\om,p)$ (Theorem \ref{thm1}) whence the monotonicity of $\mul$
is not trivial at all. However $\mul\searrow 0^+$ as $\lm\searrow 0^+$ (in which case $\all\nearrow \al_{0}^-=1^-$) and
$\mul\searrow 0^+$ as $\lm\nearrow \lm_+(\om,p)$ (in which case $\all\searrow 0^+$). Of course this is consistent
with the above mentioned classical results in \cite{Lions1982existence} for $\om$ convex.

\noi Interestingly enough, again for $\om=\mathbb{D}_{N}$,  we catch the sharp bending shape of solutions of {\rm $\pqm$}, showing that,
at least in this case, this new method is effective at least as known techniques for radial problems (\cite{Korman2012global}).
\begin{thm}\label{thm4}
Let $(\all,\pl)$ be the unique solutions of {\rm $\prl$} on $\om=\mathbb{D}_{N}$ for~$0<\lm<\lm_+(\mathbb{D}_{N},p)$ and denote~$(\mul,\vla)=(\lm\all^{p-1},\frac{\lm}{\all}\pl)\in \mathcal{R}_+$.

Then~$\vla$ solves {\rm $\pqm$} for~$\mu=\mul$, and~$\mathcal{R}_+$ is a real analytic curve defined in~$\lm\in [0,\lm_+(\mathbb{D}_N,p) )$.

Furthermore, there exists $\lm^{t}\in (0,\lm_+(\mathbb{D}_{N},p))$ such that
\begin{itemize}
    \item $\frac{d\mul}{d\lm}>0$ for $\lm\in[0,\lm^{t})$,
    \item $\frac{d\mul}{d\lm}<0$ for $\lm\in(\lm^{t},\lm_+(\mathbb{D}_{N},p))$,
     \item $\frac{d \el}{d\lm}>0$ for $\lm\in[0,\lm_+(\mathbb{D}_{N},p))$,
     \item $\mul \searrow 0^+$ and $\el \to E_0$ as $\lm\to 0^+$,
     \item $\mul \searrow 0^+$ and
     $\el \to E_\ii(\mathbb{D}_N,p)\in (E_0,+\ii)$ as $\lm\nearrow \lm_+(\mathbb{D}_{N},p)$.
\end{itemize}
\end{thm}
\noindent
Actually, due to a sharp estimate in \cite{BJ2022new},
 limited to the case $N=2$ we have that $E_{\infty}= \frac{p+1}{16\pi}$. It is also worth to recall that we have the explicit expression of $\lm_+(\mathbb{D}_{N},p)$ in terms of Sobolev constants, see Theorem B and \eqref{lmpN} above.
The proof is based on the fact that $\mul^{\frac{1}{p-1}}$ admits a natural
parametrization in terms of the unique solution of \eqref{Emden0}. On one side this fact, as well as the proof of Theorem \ref{thm:lm*}, heavily rely on the radial symmetry of the
solutions, which, as mentioned at the very beginning of this introduction, unfortunately is exactly what we wanted to circumvent.
 On the other side, once we get to this point, compared to known results the proof of the monotonicity of $\mul$
is surprisingly simple and natural. Moreover, some evaluations based on the Pohozaev identity seem to suggest an alternative proof of Theorem \ref{thm4},
 which could be hopefully flexible enough to attack the monotonicity of $\mul$ at least on convex and symmetric domains in dimension $N=2$.\\

The qualitative behavior of the corresponding branch of solutions in the $(\mu,E)$ plane is depicted in Figure~\ref{fig1}.
As remarked above, modulo the replacement of $v_\mu(0)$ with the energy,
Theorem \ref{thm4} is saying essentially the same as that the above mentioned results in \cite{Korman2012global} say about
the Rabinowitz continuum of $\pqm$ on $\mathbb{D}_{N}$. Remark that the curve bends back to a finite limit ($E_\infty$) on the energy axis, which does not correspond to a non trivial solution of $\pqm$, but rather to the fate of the indeterminate form as
$\all\to 0^+$ in \eqref{enrg}. A natural question is whether or not $\mul$ is a strictly concave function, a problem which will be addressed elsewhere.\\

\begin{figure}[h]
\psfrag{L}{$\mu^t$}
\psfrag{V}{$\mu$}
\psfrag{U}{$E$}
\psfrag{M}{$E_\infty$}
\psfrag{Z}{$(0,E_0)$}
\psfrag{E}{$E_{\lm^t}$}
\psfrag{K}{$(0,0)$}
\psfrag{A}{$(\mu_\lm,\el)$}
\includegraphics[totalheight=2in, width=3in]{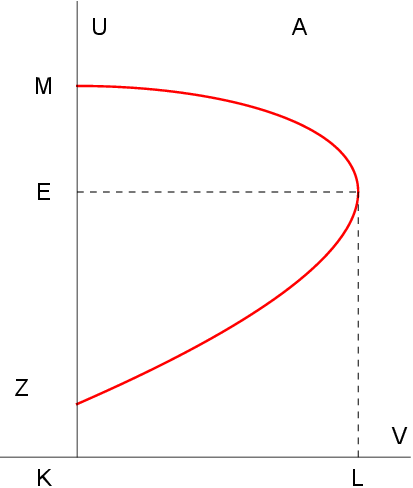}
\caption{The graph of $(\mul,\el),\,\lm\in [0,\lm_+(\mathbb{D}_{_N},p)]$. Here $\mu^t=\mu_\lm\left.\right|_{\lm=\lm^t}$.}\label{fig1}
\end{figure}


Many rather natural related problems seem to be interesting, such as the generalization of these results to more general operators
and nonlinearities, including in particular the full subcritical range of growths $p_{_N}\leq p< \frac{N+2}{N-2}$.
Concerning this point we do not expect an easy adaptation, since the assumption $p<p_{_N}$ is readily seen to play a crucial role, at least as far as we are concerned with the proof pursued here.

\bigskip
\bigskip

\section{Preliminary estimates for constrained solutions}
Since $|\om|=1$, by the constraint in $\prl$ we have that $\all\leq 1$ for any non-negative solution. Here we have,
\begin{lemma}\label{lemE1}
    Let $p\in[1,p_{_N})$. For any $\ov{\lm}>0$ there exist $\ov{\al}=\ov{\al}(r,\om,\ov{\lm},p,N)>-\ii$ and $C_1=C_1(r,\om,\ov{\lm},p,N)<+\infty$ such that
    \begin{align}
        \all\geq \ov{\al}, & & \|\pl\|_{C^{2,r}_0(\ov{\om})}\leq C_1
    \end{align}
    for any solution~$(\all,\pl)$ of {\rm $\prl$} with $\lm\in [0,\ov{\lm}\,]$.
\end{lemma}
This extends the estimates for positive solutions in~\cite[Lemma 4.2]{BJ2022uniqueness} and will be crucial for later bifurcation analysis.

\begin{proof}[Proof of Lemma~\ref{lemE1}]
    We can assume w.l.o.g. $\all<0$, otherwise $0\leq \all\leq 1$ and the proof for \cite[Lemma 4.2]{BJ2022uniqueness} works fine.
    Because of the constraint~$\ino [\all+\lm\pl]_+^p=1$, it is well-known (\cite{Littleman1963regular}) that for any $t\in [1,\frac{N}{N-1})$ there exists $C=C(t,N,\om)$ such that $\|\pl\|_{W_0^{1,t}(\om)}\leq C(t,N,\om)$ for any solution of $\prl$.
    Thus, by the Sobolev inequalities, for any $1\leq s< p_N=\frac{N}{N-2}$ we have $\|\pl\|_{L^s(\om)}\leq C(s,N,\om)$, for some $C(s,N,\om)$.
    As a consequence, since $p<p_{N}$, we have
    \begin{align}
        \ino [\all+\lm\pl]_+^{pm}\leq \lm^{pm}\ino \pl^{pm}\leq \ov{\lm}^{pm} C(m,N,\om)
    \end{align}
    for any $1<m<\frac{p_N}{p}$, and $\lm\leq \ov{\lm}$.
    By standard elliptic estimates, $\|\pl\|_{W_0^{2,m}(\om)}\leq C_0(p,N,\ov{\lm},\om)$, for any $\lm\leq \ov{\lm}$.
    Thus by the Sobolev embedding we find $\|\pl\|_{W_0^{1,t}(\om)}\leq C(t,N,\om)$ for any $t\in [1,\frac{Nm}{N-m}]$.

    At this point, since~$\frac{Nm}{N-m}>\frac{N}{N-1}$ and~$\om$ is of class $C^{2,r}$, by a bootstrap argument we find that
    \begin{align}
     \|\pl\|_{C^{2,r}_0(\ov{\om})}\leq C_1(r,\om,\ov{\lm},p,N)
    \end{align}
    for any $\lm\leq \ov{\lm}$.
    As a consequence, if $\all$ were not uniformly bounded from below, we would have $\al_n\to -\ii$ along some sequence $\lm_n\leq \ov{\lm}$, which would be a contradiction to $\ino [\al_n+\lm_n\psi_n]_+^p=1$.
\end{proof}

\


\section{Spectral setup of the linearized problem}\label{sec3}

For fixed $\lm\geq 0$ and $p\in[1,p_{_N})$, let $(\all,\pl)$ be a solution of $\prl$,  we will adopt the following notation,
$$
\Abracket{\phi}_{\ssl}=\dfrac{\ino \rlq  \phi}{\ino \rlq}\quad \mbox{and }\quad
[\phi]_{\ssl}=\phi \,-\Abracket{\phi}_{\ssl},\quad \phi \in L^2(\om).
$$
Let us recall that we are always concerned with classical solutions, whence in particular
\begin{align}
    \rlq=[\all+\lm\pl]_+^{p-1}\in C^0(\om).
\end{align}
The support of $\rl$ is by definition $\ov{\om_+}$,
whence, if~$\all\geq 0$ then~$\om_+= \om$ while if~$\all<0$ then, since $\pl$ is continuous, we have $\om_+\Subset \om$.\\

Consider the linear operator~$L_{\ssl}\colon C^{2,r}_0(\ov{\Omega})\to C^r(\ov{\Omega})$ defined by
\begin{align}\label{eLl}
    L_{\ssl}[\phi]=-\Delta \phi-\tl \rlq [\phi]_{\ssl}
\end{align}
with~$\tl=\lm p$.
We say that $\sg=\sg(\all,\pl)\in\R$ is an eigenvalue of $L_{\ssl}$ if the equation
\begin{align}\label{lineq0.1}
-\Delta \phi-\tl \rlq [\phi]_{\ssl}=\sg\rlq [\phi]_{\ssl},
\end{align}
admits a non-trivial weak solution $\phi\in H^1_0(\om)$, and the corresponding eigenspace is denoted by $\Eigen(L_{\ssl};\sg)$.
Note that~$L_{\ssl}$ involves not only nonlocal terms but also a weight which may vanish on a large set.
This is in great contrast to~\cite{BJ2022uniqueness}; although the structure of the spectral theory goes along the same ideas
in~\cite{BJ2022uniqueness}, a relevant
part of the argument differs essentially.
Now it is crucial to understand the spectral properties of $L_{\ssl}$.
Although the main properties for general weighted operators are considered in~\cite{ManesMicheletti1973estensione},
here we need a refined information about the spectrum and, in particular, about the eigenfunctions.

The eigenvalue equation~\eqref{lineq0.1} can be written as follows,
\begin{align}\label{eq:weighted eigenvalue}
    -\Delta \phi = (\tl+\sigma)\rlq [\phi]_{\ssl}.
\end{align}
If we let~$G$ denote the Green's function for~$-\Delta$ with Dirichlet boundary condition and apply it to~\eqref{lineq0.1}, then we see that,
\begin{align}
    \phi =  \frac{\tl+\sg}{\tl}\;  G* \parenthesis{\tl\rlq[\phi]_{\ssl}}.
\end{align}
As in~\cite{BJ2022uniqueness} we consider the operator~$T_{\ssl} \colon H^1_0(\Omega)\to H^1_0(\Omega)$ defined as follows,
\begin{align}
    T_{\ssl}(\phi) \coloneqq G* \parenthesis{\tl\rlq[\phi]_{\ssl}},
\end{align}
so that~\eqref{lineq0.1} is equivalent to
\begin{align}\label{eq:eigenvalue for T}
    T_{\ssl}(\phi)= \mu\phi, \qquad \mbox{ with } \mu = \frac{\tl}{\tl+\sg}.
\end{align}
In other words,~$\phi\in \Eigen(L_{\ssl};\sigma)$ iff~$\phi\in \Eigen(T_{\ssl};\frac{\tl}{\tl+\sg})$.
Thus it suffices to understand the eigenvalues and eigenfunctions of~$T_{\ssl}$, where the advantage is that it is a linear self-adjoint compact operator on the Hilbert space~$H^1_0(\Omega)$ equipped with the inner product
\begin{align}
    \Abracket{\xi,\eta}_{H^1_0}=\int_{\Omega} \Abracket{\nabla\xi,\; \nabla\eta}, \qquad  \forall \xi,\eta\in H^1_0(\Omega).
\end{align}
Indeed,
\begin{align}
    \Abracket{T_{\ssl}(\xi),\; \eta}_{H^1_0}
    =& \int_{\Omega} \Abracket{\nabla T_{\ssl}(\xi), \nabla\eta}
    =\int_{\Omega} (-\Delta T_{\ssl}(\xi))\eta\\
    =&\int_{\Omega} \tl\rlq [\xi]_{\ssl} \eta
    =\int_{\Omega} \tl\rlq [\xi]_{\ssl} [\eta]_{\ssl}
\end{align}
which is symmetric in~$\xi$ and~$\eta$, hence~$T_{\ssl}$ is self-adjoint.
The compactness of~$T_{\ssl}$ follows immediately from the compactness of~$G$.
As a consequence, the spectrum of~$T_{\ssl}$, say~$\Spect(T_{\ssl})$, consists of countably many real nonzero eigenvalues and perhaps zero.
Each nonzero eigenvalue has finite multiplicity, and the nonzero eigenvalues have zero as the unique accumulation point.
From~\eqref{eq:weighted eigenvalue} it follows that~$\tl+\sigma\geq 0$, hence~$\Spect(T_{\ssl})\subset\R_+$.
We may thus list the nonzero eigenvalues as follows,
\begin{align}
    \mu_1\geq \mu_2\geq \mu_3\geq \cdots >0, \quad \lim_{j\to +\infty}\mu_j = 0,
\end{align}
and denote the corresponding eigenfunctions by~$\phi_j$,~$j\in\mathbb{N}$, which satisfy
\begin{align}
    \delta_{jk}=\Abracket{\phi_j,\phi_k}_{H^1_0}=\int_\Omega \Abracket{\nabla \phi_j, \;\nabla\phi_k}, \qquad \forall j,k\geq 1.
\end{align}
Introducing the bilinear form,
\begin{align}
    \mathcal{B}(\xi,\eta)\coloneqq \tl\int_{\om} \rlq [\xi]_{\ssl}[\eta]_{\ssl}, \qquad \forall \xi,\eta\in H^1_0(\Omega),
\end{align}
then it is well known (see \cite{ManesMicheletti1973estensione}) that the eigenvalues can be characterized as follows :
\begin{itemize}
    \item for~$j=1$:
            \begin{align}
                \mu_1= &\max \braces{ \mathcal{B}(\phi,\phi) \mid \phi\in H^1_0(\Omega),\; \Abracket{\phi,\phi}_{H^1_0}=1}  \\
                =&\; \mathcal{B}(\phi_1,\phi_1),
            \end{align}
    \item for~$j\geq 2$:
            \begin{align}
                \mu_j=& \max\braces{\mathcal{B}(\phi,\phi)\mid \phi\in H^1_0(\Omega),\; \Abracket{\phi,\phi}_{H^1_0}=1, \; \Abracket{\phi,\phi_k}_{H^1_0} =0, \mbox{ for } 1\leq k\leq j-1} \\
                =& \;\mathcal{B}(\phi_j,\phi_j).
            \end{align}
\end{itemize}
We collect the set of nonzero eigenvalues in~$\Spect(T_{\ssl})\setminus \braces{0}$ and denote
\begin{align}
    \mathcal{H}_1\coloneqq \ov{\oplus_{j\geq 1} \Eigen(T_{\ssl};\mu_j) } = \ov{\Span\braces{\phi_j\mid j\geq 1}}
\end{align}
where the closure is taken in the~$H^1_0$ norm.
This is a closed subspace of~$H^1_0(\Omega)$.

\begin{lemma}\label{lem:spectral}$\,$
    \begin{itemize}
        \item[(i)] $0\in\Spect(T_{\ssl})$ iff~$\alpha<0$;
        \item[(ii)] If~$\alpha\geq 0$, then
                    \begin{align}
                        \Spect(T_{\ssl})=\braces{\mu_j\mid j\geq 1}\quad \mbox{ and }\quad
                        H^1_0(\Omega)=\mathcal{H}_1;
                    \end{align}
        \item[(iii)] If~$\alpha<0$, then
                \begin{align}
                    \Spect(T_{\ssl})=\braces{0}\cup \braces{\mu_j\mid j\geq 1}\quad  \mbox{ and }\quad
                    H^1_0(\Omega)= \Eigen(T_{\ssl};0)\oplus\mathcal{H}_1.
                \end{align}
                Moreover,
                \begin{align}\label{eq:og decomp}
                    \Eigen(T_{\ssl};0)=\braces{\phi\in H^1_0(\Omega)\mid \phi|_{\Omega_+} = \mbox{const} }
                \end{align}
                and the restriction of  any function in~$\mathcal{H}_1$ to the set~$\Omega\setminus\ov{\Omega_+}$ is harmonic.
    \end{itemize}
\end{lemma}

    \begin{proof}
        (i) Taking the Laplacian of~$T_{\ssl}(\phi)= \mu\phi$ gives
        \begin{align}
            -\mu\Delta\phi= \tl\rlq[\phi]_{\ssl}
        \end{align}
        and testing against~$\phi$ we have that
        \begin{align}
            \mu\int_{\Omega} |\nabla\phi|^2 = \tl\int_{\Omega} \rlq [\phi]_{\ssl}^2 .
        \end{align}
        If~$\alpha\geq 0$, then~$\Omega_+= \Omega$ by the maximum principle and
        if~$\mu=0$ were an eigenvalue with an eigenfunction~$\phi\neq 0$, then~$[\phi]_{\ssl}\equiv 0$, whence~$\phi=\Abracket{\phi}_{\ssl}$ would be constant, which should be necessarily zero since~$\phi\in H^1_0(\Omega)$. This is the desired contradiction, showing the first implication.

        Conversely, if~$\alpha<0$ then~$\Omega\setminus \ov{\Omega_+}$ is a nonempty open subset and any~$0\neq \phi \in H^1_0(\Omega\setminus \ov{\Omega_+})$ is an eigenfunction of~$T_{\ssl}$.

        (ii) If~$\alpha\geq 0$, then we see that~$\Spect(T_{\ssl})=\braces{\mu_j\mid j\geq 1}$ doesn't contain~$0$, and the normalized eigenfunctions constitutes a complete orthonormal basis for~$H^1_0(\Omega)$ by standard results.

        (iii) If~$\alpha<0$, then~$\Spect{T_{\ssl}}=\braces{0}\cup\braces{\mu_j\mid j\geq 1}$, and~$H^1_0(\Omega)=\Eigen(T_{\ssl};0)\oplus \mathcal{H}_1$.
        Since we have seen that~$H^1_0(\Omega\setminus\ov{\Omega_+})\subset\Eigen(T_{\ssl};0)$, the eigenvalue~$0$ has infinite multiplicity.

        Let~$\phi_0\in \Eigen(T_{\ssl};0)$, namely,
        \begin{align}
            0=T_{\ssl}(\phi_0)=G*\parenthesis{\tl\rlq [\phi_0]_{\ssl}},
        \end{align}
        hence~$\rlq[\phi_0]_{\ssl}\equiv 0$ a.e. in~$\Omega$.
        It follows that~$[\phi_0]_{\ssl}\equiv 0$ a.e. in~$\Omega_+$ or equivalently,~$\phi_0 = \Abracket{\phi_0}_{\ssl}$ is constant a.e. in~$\Omega_+$.
        Again, since~$H^1_0(\Omega\setminus\ov{\Omega_+})\subset\Eigen(T_{\ssl};0)$,~$\phi_0$ can be arbitrary in~$\Omega\setminus\ov{\Omega_+}$ as long as it is~$H^1$ and has trace~$0$ on~$\partial\Omega$.

        For~$\mu_j>0$, the eigenfunction~$\phi_j$ satisfies
        \begin{align}
            -\Delta\phi_j =\frac{1}{\mu_j} \tl\rlq [\phi_j]_{\ssl}
        \end{align}
        and the right hand side vanishes in~$\Omega\setminus\Omega_+$, namely~$\phi_j|_{\Omega\setminus\ov{\Omega_+}}$ is necessarily harmonic.
    \end{proof}

\begin{rmk}\label{rem:orth}
    In the orthogonal decomposition, we will use
    \begin{align}
        P_0\colon H^1_0(\Omega)\to\Eigen(T_{\ssl};0) \quad  \mbox{ and }\quad
        P_1\colon H^1_0(\Omega)\to \mathcal{H}_1
    \end{align}
    to denote the corresponding orthogonal projections.
    More precisely, for any~$\psi\in H^1_0(\Omega)$, we have that
    \begin{align}
        \psi= P_0 \psi + P_1 \psi = P_0\psi + \sum_{j=1}^{+\infty} \beta_j\phi_j
    \end{align}
    with the Fourier coefficients {\rm
    \begin{align}
        \beta_j
        = \Abracket{\psi,\phi_j}_{H^1_0}
        =\ino \Abracket{\nabla\psi,\;\nabla\phi_j}
        =\ino \frac{\tl}{\mu_j} \rlq [\phi_j]_{\ssl} \psi
        =\ino \frac{\tl}{\mu_j} \rlq [\phi_j]_{\ssl} [\psi]_{\ssl}.
    \end{align}}

\end{rmk}

Recall that each~$\mu_j>0$ of~$T_{\ssl}$ corresponds to an eigenvalue~$\sg_j$ of~$L_{\ssl}$, related by
\begin{align}
    \mu_j=\frac{\tl}{\tl+\sg_j}, \quad \mbox{ that is, }\quad  \sg_j= \tl\parenthesis{\frac{1}{\mu_j}-1},
\end{align}
sharing the same eigenfunction~$\phi_j$.
The zero eigenvalue of~$T_{\ssl}$ does not correspond to any eigenvalue of~$L_{\ssl}$.
Indeed,~$L_{\ssl}|_{\Eigen(T_{\ssl};0)}= (-\Delta)|_{\Eigen(T_{\ssl};0)}$, which gives the other part of the spectrum of~$L_{\ssl}$.

Note that~$0\in\Spect(L_{\ssl})$ iff~$1\in\Spect(T_{\ssl})$.

\begin{lemma}\label{lem:iso}
If $0\notin \Spect(L_{\ssl}) $, then for $p>1$, $L_{\ssl}\colon C^{2,r}_0(\ov{\Omega})\to C^r(\ov{\Omega})$ is an isomorphism, while for $p=1$, 
$L_{\ssl}\colon W^{2,t}_0(\Omega)\to L^t(\Omega)$ is an isomorphism for some $t>N$.
\end{lemma}
\begin{proof}
    The assumption~$0\notin \Spect(L_{\ssl})$ guarantees that~$L_{\ssl}$ is injective. We discuss only the case $p>1$ just to avoid technicalities.
    For the surjectivity, let~$f\in C^\beta(\ov{\Omega})$, and we need to find a solution of
    \begin{align}
        L_{\ssl}\varphi= -\Delta \varphi-\tl\rlq [\varphi]_{\ssl} =f.
    \end{align}
    Applying~$G$ to both sides, this is equivalent to
    \begin{align}
        \varphi-T_{\ssl}(\varphi) = G*f.
    \end{align}
    By projecting to the subspaces~$\Eigen(T_{\ssl};0)$ and~$\mathcal{H}_1$ respectively, say~$\varphi=\varphi_0+\varphi_1$, and~$G*f=(G*f)_0+ (G*f)_1$, with~$\varphi_0, (G*f)_0 \in \Eigen(T_{\ssl};0)$ and~$\varphi_1, (G*f)_1\in \mathcal{H}_1$, and noting that~$T_{\ssl}(\varphi_0)=0$, we get
    \begin{align}\label{eq:surjectivity of L}
        \varphi_0 =(G*f)_0, & &
        (I-T)\varphi_1= (G*f)_1.
    \end{align}
    Since~$1\notin\Spect(T_{\ssl})$, by the Fredholm alternative on the  Hilbert subspace~$\mathcal{H}_1$, we see that there exists a unique~$\varphi_1$ solving the second equation in~\eqref{eq:surjectivity of L}.
\end{proof}

\

To describe the continuous branch of solutions, we employ the implicit function theorem.
Thus consider the map
\begin{align}
    F\colon (-1,+\infty)\times \R \times C^{2,r}_0(\ov{\Omega}) &\to \R \times C^r(\ov{\Omega}), \\
            (\lambda,\alpha, \psi)& \mapsto F(\lambda,\alpha, \psi )=(F^1, F^2),
\end{align}
with
\begin{align}\label{eF}
    F^1(\lambda,\alpha,\psi)=  -1+\int_\Omega [\alpha+\lambda\psi]_+^p, & &
    F^2(\lambda,\alpha,\psi)= -\Delta\psi-[\alpha+\lambda\psi]_+^p  \in C^\beta(\ov{\Omega}).
\end{align}
The preimage~$F^{-1}(0,0)$ consists exactly of solutions of~$\prl$.
The gradient w.r.t. $(\alpha,\psi)$ is given by
\begin{align}
    D_{(\alpha,\psi)} F(\lambda,\alpha,\psi)[s,\phi]
    = \parenthesis{ p\int_\Omega [\alpha+\lambda\psi]_+^{p-1}(s+\lambda\phi), \;  -\Delta\phi- p[\alpha+\lambda\psi]_+^{p-1}(s+\lambda\phi) }.
\end{align}

\

Fix~$\lm>0$ and let~$(\all, \pl)$ be a solution of~$\prl$, so that~$F(\lm,\all,\pl)=(0,0)$.
Using the notation~$\rl$ as above, we have
\begin{align}\label{eq:differential F}
    D_{(\alpha,\psi)} F(\lm,\all,\pl)[s,\phi]
    = \parenthesis{ p\int_\Omega \rlq (s+\lambda\phi), \;  -\Delta\phi-p\rlq (s+\lambda\phi)}.
\end{align}

Consider first the kernel of~$D_{(\alpha,\psi)}F(\lm,\all,\pl)$.
Let~$(s,\phi)$ be an element in the kernel.
The vanishing of the first component in~\eqref{eq:differential F} implies~$ s=-\lm\Abracket{\phi}_{\ssl}$, hence the vanishing of the second component becomes
\begin{align}
    0=-\Delta\phi -\tl \rlq [\phi]_{\ssl}=L_\lm \phi.
\end{align}
If~$\phi\neq 0$, then~$0\in \Spect(L_\lm)$.
That is, \emph{if~$0\notin \Spect(L_\lm)$, then~$\phi=0$ and hence also~$s=0$, so~$\Ker D_{(\alpha,\psi)} F(\lambda,\all,\pl)=0$. }

Given~$(t,f)\in \R\times C^r(\ov{\Omega})$, consider the equation
\begin{align}
    D_{(\alpha,\psi)} F(\lm,\all,\pl)[s,\phi]= (t,f).
\end{align}
Similarly, the relation for the first components gives
\begin{align}
     s = -\lm\Abracket{\phi}_{\ssl} + \frac{t}{p\ml}, \qquad (\mbox{ where } \ml \equiv  \ino\rlq \;)
\end{align}
while the second becomes
\begin{align}
    -\Delta\phi- \tl \rlq [\phi]_{\ssl} = f+\frac{t}{\ml}\rlq.
\end{align}
By Lemma~\ref{lem:iso}, there is a unique~$\phi\in C^{2,r}_0(\ov{\Omega})$ solving this equation, which in turn determines a unique~$s$.
Thus~$L_{\ssl}$ is also surjective, hence an isomorphism.

\

One can readily check that the above argument also holds for~$\lm=0$.

Summarizing these facts, together with the implicit function theorem, we have the following local~$C^1$ regularity of the branch of solutions.
See~\cite{AmbrosettiProdi1993primer} and compare with~\cite[Lemma 2.4]{BJ2022uniqueness}.

\begin{lemma}\label{lem1.1}
Let $p\in (1,p_N)$, $(\al_{\sscp \lm_0},\psi_{\sscp \lm_0})$ be a solution of {\rm $\prl$} with $\lm=\lm_0\geq 0$.
If $0\notin\Spect(L_{\sscp \lm_0})$, then:
\begin{itemize}
    \item[(i)] $D_{\al,\psi}F(\lm_0, \al_{\sscp \lm_0}, \psi_{\sscp \lm_0})$ is an isomorphism;
    \item[(ii)] There exists an open neighborhood $\mathcal{U}$ of $(\lm_0,\al_{\sscp \lm_0},\psi_{\sscp \lm_0})$ such that the set of solutions of {\rm $\prl$} in $\mathcal{U}$ is a $C^1$ curve of solutions $J\ni\lm\mapsto (\all,\pl)\in B$, for suitable neighborhoods $J$ of $\lm_0$ and $B$ of $(\al_{\sscp \lm_0},\psi_{\sscp \lm_0})$ in $\R\times C^{2,r}_{0}(\ov{\om}\,)$.
\end{itemize}
If $p=1$ the same is true where $F:(-1,+\infty)\times \R \times W^{2,t}_0({\Omega}) \to \R \times L^t({\Omega})$, for some $t>N$.
\end{lemma}

Let us equivalently write the characterization of~$\mu_1$ as follows,
\begin{align}
    \frac{1}{\mu_1}=\inf \left\{\frac{1}{\mathcal{B}(\phi,\phi)}\,:\,\phi\in H^1_0(\om),\; \Abracket{\phi,\phi}_{H^1_0}=1\right\},
\end{align}
then it is readily seen that
\begin{align}\label{4.1}
\sg_1=\sg_1(\all,\pl)=\tl\left(\frac{1}{\mu_1}-1\right)=\inf\limits_{\phi \in H^1_0(\om)\setminus \{0\}}
\dfrac{\ino |\nabla \phi|^2 - \tl \ino \rlq [\phi]_{\ssl}^2 }{\ino \rlq [\phi]_{\ssl}^2}.
\end{align}
On the other hand, the standard first eigenvalue is characterized by
\begin{align}
     \nu_{1}(\all,\pl)\coloneqq \inf\limits_{w\in H^{1}_0(\om)\setminus\{0\}}
        \dfrac{\ino |\nabla w|^2-\tl \ino \rlq w^2  }{\ino \rlq w^2}.
\end{align}

\begin{prop}\label{preig}
    Let~$(\all,\pl)$ is a solution of~{\rm $\prl$}.
    Then
    \begin{itemize}
        \item[(i)] $\nu_1(\all,\pl)\geq \Lambda(\om,2p)-\lm p$,
        \item[(ii)] $\sg_1(\all,\pl)>\nu_1(\all,\pl)$.
    \end{itemize}
    In particular, if~$\lambda p \leq \Lambda(\Omega, 2p)$, then~$\sg_1(\all,\pl)>\nu_1(\all,\pl)\geq 0$.

\end{prop}

    \begin{proof}
        (i) The argument is similar to that in~\cite{BJ2022uniqueness}, only noting that here~$\rl$ can vanish in a large subset but still integrates to 1.
        W.l.o.g. we consider~$\lm>0$.
        Let $w \in C^{1}_0(\ov{\om}\,)$, $w\equiv \!\!\!\!\!/ \;0$ then we have
        \begin{align}
        \dfrac{\ino |\nabla w|^2}{\ino \rlq w ^2}\geq \dfrac{1}{ \left(\ino [\all+\lm\pl]_+^p\right)^{\frac1q}}
        \dfrac{\ino |\nabla w|^2}{\left(\ino w^{2p}\right)^{\frac1p}}=
        \dfrac{\ino |\nabla w|^2}{\left(\ino w^{2p}\right)^{\frac1p}}\geq {\Lambda(\om,2p)}
        \end{align}
        which implies that~$\nu_{1}(\all,\pl)\geq {\Lambda(\om,2p)}-\lm p$.

        (ii) Clearly $\sg_1(\all,\pl)\geq \nu_{1}(\all,\pl)$.
        Argue by contradiction and assume that $\sg_1(\all,\pl)=\nu_{1}(\all,\pl)$, then there exists some $\{\phi,w\} \in H^1_0(\om)\setminus \{0\}$ such that
            $$
            \sg_1(\all,\pl)=\dfrac{\ino |\nabla \phi|^2 - \tl \ino \rlq [\phi]_{\ssl}^2 }{\ino \rlq [\phi]_{\ssl}^2}=
            \dfrac{\ino |\nabla w|^2-\lm p \ino \rlq w^2  }{\ino \rlq w^2}=\nu_{1}(\all,\pl)
            $$
            However, if $\Abracket{\phi}_{\ssl}\neq 0$, then
            $$
            \nu_{1}(\all,\pl)\leq \dfrac{\ino |\nabla \phi|^2 - \tl \ino \rlq \phi^2 }{\ino \rlq \phi^2}<\sg_1(\all,\pl),
            $$
            which is absurd.
            Thus we have $\Abracket{\phi}_{\ssl}= 0$, which is also impossible: as an eigenfunction of $\nu_{1}(\all,\pl)$, ~$\phi$ should have a sign, contradicting~$\Abracket{\phi}_{\ssl}=0$.
    \end{proof}

\bigskip

\section{Monotone parametrization of solutions}\label{sec3.4}
In this section we are concerned with the monotonicity of the energy
$$
2 E_{\ssl}=\ino \rl \pl,
$$
and of $\all$ for solutions of $\prl$ with $\sg_1(\all,\pl)>0$ (see \eqref{4.1}).
The energy monotonicity is detected by making use of the spectral setting in the previous section.
We recall that $\ml=\ino \rlq$ and $L_{\ssl}$ is the linearized operator defined in \eqref{eLl}.
Recall by Remark \ref{rem:orth} that $P_1\colon H^1_0(\Omega) \to \mathcal{H}_1$ denotes the orthogonal projection on the the closed subspace~$\mathcal{H}_1$.

\begin{prop}\label{pr-enrg}
    Let $(\al_{\ssl_0},\psi_{\ssl_0})$ be a solution of {\rm $\prl$} with $\lm=\lm_0\geq 0$ and suppose that $0\notin\Spect(L_{\ssl_0})$.
    Then, locally near $\lm_0$, the map $\lm\mapsto (\all,\pl)$ is a $C^1$ simple curve and {\rm
    \begin{align}
        2 E_{\ssl}=\ino \rl \pl,  \qquad \vl\coloneqq \dfrac{\dd \pl}{\dd \lm}
    \end{align}}
    are $C^1$ functions in $\lm$ with $\vl \in C^{2}_0(\ov{\om})$ for $p>1$, $\vl \in W^{2,t}_0({\om})$ for some $t>N$ for $p=1$, and {\rm
    \begin{align}\label{8.12.11}
        \frac{\dd E_{\ssl}}{\dd \lm} = \ino \rl \vl=\tl \ino\rlq[\vl]_{\ssl}[\pl]_{\ssl} +  p \ino  [\pl]_{\ssl}^2.
    \end{align}}

   If, in addition, $\sg_1=\sg_1(\all,\pl)> 0$, then  {\rm
    \begin{align}\label{21.06.17}
        \int_\Omega \rlq [\vl]_{\ssl} [\pl]_{\ssl}
        \geq\sg_1\frac{\ml}{p} \Abracket{[P_1(\vl)]_{\ssl}^2}_{\lm},
    \end{align}}
    and in particular $\frac{d \el}{d\lm}>0$.
\end{prop}

The corresponding statement was obtained for positive solutions in~\cite{BJ2022uniqueness}.
For solutions which are not necessarily positive, more care is needed for the spectral analysis.
We already see that there is a new subspace~$\Eigen(T_{\ssl};0)$ appearing in this context.
It turns out that this will not affect the energy estimates since the derivative of the energy “does not see” the~$\Eigen(T_{\ssl};0)$ subspace.
This will be clear in the proof and is the crucial point for the argument to work.

\begin{proof}[Proof of Proposition~\ref{pr-enrg}]
    That~$(\all,\pl)$ is~$C^1$ in~$\lm$ locally around~$(\al_{\ssl_0},\psi_{\ssl_0})$ follows from Lemma~\ref{lem1.1}. 
    Thus~$\vl \in C^{2,r}_0(\ov{\om})$ for $p>1$, $\vl \in W^{2,t}_0({\om})$ for some $t>N$ for $p=1$ and satisfies
    \begin{align}
        -\Delta \vl =\tl \rlq \vl +p \rlq\pl +p\rlq \frac{\dd\all}{\dd\lm},
    \end{align}
    where~$\frac{\dd\all}{\dd\lm}$ can be computed by the unit mass constraint in $\prl$, that is
$$
p\frac{\dd\all}{\dd\lm}=-\tl  \Abracket{\vl}_{\ssl}-p\Abracket{\pl}_{\ssl}.
$$
Therefore, we conclude that $\vl\in H^1_0(\Omega)$ is a solution of,
\begin{align}\label{1b1}
-\Delta \vl =\tl \rlq [\vl]_{\ssl} +p \rlq[\pl]_{\ssl}.
\end{align}
 
By using $\prl$, we also have $E_{\ssl}=\frac12\ino |\nabla \pl|^2$, and in particular it holds,
$$
\frac{\dd E_{\lm}}{\dd \lm} =\ino \Abracket{\nabla \vl,\nabla \pl}=\ino \vl (-\Delta \pl)
=\ino \rl\vl,
$$
which proves the first equality in \rife{8.12.11}.
Meanwhile integrating by parts and using~\eqref{1b1} shows that,
\begin{align}\label{2b1}
\ino\rl \vl=\ino (-\Delta \pl)\vl=\ino \pl(-\Delta \vl)= \tl \ino\rlq[\vl]_{\ssl}\pl +  p \ino \rlq [\pl]_{\ssl}\pl,
\end{align}
which is the second equality in \rife{8.12.11}.

It remains to prove the monotonicity of~$E_{\lm}$.
The case~$\all\geq 0$ has been already discussed in~\cite[Proposition 3.1]{BJ2022uniqueness}, so we may assume~$\all<0$.
We use the decomposition in Lemma~\ref{lem:spectral} and Remark~\ref{rem:orth} and, recalling that
    \begin{align}
        1=\ino \Abracket{\nabla\phi_j,\nabla\phi_j}
        = \ino \frac{\tl}{\mu_j} \rlq [\phi_j]_{\ssl}^2,
    \end{align}

write
\begin{align}
    \pl=P_0(\pl)+\sum\limits_{j=1}^{+\ii}\beta_j\phi_j, & &
    \mbox{ with }
    \beta_j=\frac{\tl}{\mu_j}\ino \rlq[\phi_j]_{\ssl}\pl
    = \frac{\tl}{\mu_j}\ino \rlq[\phi_j]_{\ssl}[\pl]_{\ssl},
\end{align}
\begin{align}
    \vl=P_0(\vl)+\sum\limits_{j=1}^{+\ii}\gamma_j\phi_j,
    & & \mbox{ with }
    \gamma_j=\frac{\tl}{\mu_j}\ino \rlq [\phi_j]_{\ssl}\vl
    =\frac{\tl}{\mu_j} \ino \rlq [\phi_j]_{\ssl}[\vl]_{\ssl}.
\end{align}

Testing~\eqref{1b1} against~$\phi_j$ and combining with~\eqref{lineq0.1} we get that,
\begin{align}\label{lamq31}
\sg_{j}\ino\rlq[\phi_j]_{\ssl}[\vl]_{\ssl} =p
\ino\rlq [\phi_j]_{\ssl}[\pl]_{\ssl}, \qquad (\mbox{here } \sg_j\equiv \sg_j(\all,\pl))
\end{align}
that is,~$\sg_j\gamma_j=p{\beta_j}$,
which hold true for each~$j\geq 1$.
The crucial observation is that~$\rlq[P_0(\pl)]_{\ssl}\equiv 0$ and~$\rlq [P_0(\vl)]_{\ssl}\equiv 0$, hence
\begin{align}
    p\ino\rlq[\vl]_{\ssl}[\pl]_{\ssl}
    &= p\ino\rlq \parenthesis{ [P_0(\vl)]_{\ssl} + [P_1(\vl)]_{\ssl} } \parenthesis{ [P_0(\pl)]_{\ssl} + [P_1(\pl)]_{\ssl}} \\
    &= p\ino \rlq [P_1(\vl)]_{\ssl} [P_1(\pl)]_{\ssl} \\
    &= p\ino \rlq \parenthesis{\sum_{j\geq 1} \gamma_j [\phi_j]_{\ssl}} \parenthesis{\sum_{k\geq 1} \beta_k [\phi_k]_{\ssl}}  \\
    &= p\sum_{j\geq 1}\gamma_j \beta_j \ino \rlq [\phi_j]_{\ssl}^2\\
    &= \sum_{j\geq 1} p \gamma_j\beta_j \frac{\mu_j}{\tl} \\
    &= \sum_{j\geq 1} \frac{\mu_j}{\tl} \sg_j \gamma_j^2.
\end{align}
The desired estimate in~\eqref{21.06.17} follows from
\begin{align}
    \Abracket{[P_1(\vl)]_{\ssl}^2}_{\ssl}
    =& \frac{1}{\ml} \ino \rlq \parenthesis{\sum_{j\geq 1} \gamma_j[\phi_j]_{\ssl}}^2
    =\frac{1}{\ml} \sum_{j\geq 1} \frac{\mu_j}{\tl} \gamma_j^2
\end{align}
which can be estimated from above, as long as~$\sg_1>0$,
\begin{align}
    \Abracket{[P_1(\vl)]_{\ssl}^2}_{\ssl}
    \leq \frac{1}{\ml\sg_1}\sum_{j\geq 1} \frac{\mu_j}{\tl} \sg_j \gamma_j^2
    =\frac{1}{\ml\sg_1} p\ino\rlq[\vl]_{\ssl}[\pl]_{\ssl} .
\end{align}
Thus the first term in~\eqref{8.12.11} is non negative, and the right hand side of~\eqref{8.12.11} is thus positive, proving the monotonicity of~$E_\lm$.
\end{proof}

\

Next we prove the monotonicity of $\all$ whenever $\sg_1(\all,\pl)>0$.

For this and later purposes, it is convenient to denote $\xil=\lm \pl$ which satisfies
\begin{align}\label{ul.1.1}
\graf{-\Delta \xil =\lm \left[\all+\xil\right]_+^p\quad \mbox{in}\;\;\om,\\  \\
\xil=0 \;\; \mbox{on}\;\;\pa\om ,\\  \\
\bigintss\limits_{\om}  \left[\all+\xil\right]_+^p=1.
}
\end{align}
In terms of~$\xil$ we have $\rl=\left[\all+\xil\right]_+^p$ and $\rlq=\left[\all+\xil\right]_+^{p-1}$.

\begin{prop}\label{pr3.2.best}
Let $(\all,\pl)$  solve {\rm $\prl$} with $\lm\geq 0$.
If $\sg_1(\all,\pl)>0$ then $\dfrac{d \all}{d \lm}<0$.
\end{prop}

\begin{proof} For later references we provide two proofs. The first one is a straightforward consequence of Proposition \ref{pr-enrg} and Proposition \ref{pro:lat.1} below.\\
The second proof goes as follows, we first discuss the cases $p>1$ and $\lm>0$. By Lemma \ref{lem1.1} $\xil=\lm\pl$ is a $C^1$ function of $\lm$ and~$\wl=\frac{d\xil}{d\lm}\in C^{2}_0(\ov{\om}\,)$ satisfies
\begin{align}\label{wl}
\graf{-\Delta \wl =\tl \rlq [\wl]_{\ssl}+\rl\quad \mbox{in}\;\;\om\\ \\
\wl=0 \quad \mbox{on}\;\;\pa\om.
}
\end{align}
Moreover, taking derivative of the constraint~$\ino \left[\all+\xil\right]_+^p=1$, we have
\begin{align}\label{malfa.1}
\frac{d\all}{d\lm}=-\Abracket{\wl}_{\ssl}.
\end{align}

Testing \rife{wl} by $\wl$ we find that
\begin{align}\label{alm1}
\all\ino \rlq \wl+\ino \rlq \wl\xil=& \ino \rl \wl=\ino |\nabla \wl|^2-\tl \ino \rlq[\wl]^2_{\ssl}\\
\geq &  \sg_{1}(\all,\pl)\ino \rlq[\wl]^2_{\ssl}.
\end{align}
On the other side, multiplying the equation in \eqref{ul.1.1} by $\wl$ and integrating by parts
we also find that,
\begin{align}\label{030520.1}
\tl \ino \rlq [\wl]_{\ssl}\xil+\ino \rl \xil=\lm \ino \rl \wl=\lm \all \ino \rlq \wl+\lm \ino \rlq \wl{\xil},
\end{align}
which is equivalent to
$$
\ino\rlq \wl\xil=\frac{1}{p-1}(p\Abracket{\xil}_{\ssl}+\all)\ino\rlq\wl -\frac{1}{\lm(p-1)}\ino \rl \xil.
$$
Substituting this expression of $\ino\rlq \wl\xil$ in \eqref{alm1} we obtain,
$$
\tl \left(\all + {\Abracket{\xil}_{\ssl}}\right)\ino \rlq \wl-\ino \rl \xil
\geq \lm(p-1)\sg_{1}(\all,\pl)\ino \rlq[\wl]^2_{\ssl}.
$$
Since
$$
\all + {\Abracket{\xil}_{\ssl}}=\frac{\ino \rlq(\all + {\xil})}{\ml}=\frac{\ino \rl}{\ml}=\frac{1}{\ml},
$$
we thus have
$$
\tl\Abracket{\wl}_{\ssl}=\frac{\tl}{\ml} \ino \rlq \wl\geq {\lm}{(p-1)}\sg_{1}(\all,\pl)\ino \rlq[\wl]^2_{\ssl}+\ino \rl \xil.
$$

Since~$\sg_1>0$ by assumption and $\xil=\lm\pl\geq 0$ by the maximum principle, so~$\Abracket{\wl}_{\lm}>0$, hence in~\eqref{malfa.1}
we have~$\frac{d\all}{d\lm}<0$.
For $p=1$ and $\lm>0$ after a straightforward evaluation we deduce from \eqref{030520.1} that $\Abracket{\wl}_{\ssl}=2\el$ and then the conclusion follows from \eqref{malfa.1}.
For $p\geq1$ and $\lm=0$ we have $\wl=\psi_{\sscp 0}$ and $\Abracket{\wl}_{\ssl}=2E_{0}$ and the conclusion follows again from \rife{malfa.1}.
\end{proof}

At this point we can prove Theorem \ref{thm1} and Theorem \ref{thm:gel}.

\begin{proof}[Proof of Theorem~\ref{thm1}]
     We see from Proposition \ref{preig} that $\sg_1(\all,\pl)>0$ for any solution of $\prl$ for any~$0\leq \lm\leq \lm_0(\om,p)$, whence $\lm_*(\om,p)>\lm_0(\om,p)$.\\
    By Lemmas \ref{lemE1} and \ref{lem1.1} we can continue a $C^1$ curve of solutions of $\prl$ starting at $\lm=-\eps<0$ for some small $\eps>0$, for any $\lm<\lm_*(\om,p)$, which we denote by $\mathcal{G}_\eps(\om)$.
    If any other solution would exist in $(0,\lm_*(\om,p))$ which was not on $\mathcal{G}_\eps(\om)$, then by Lemma \ref{lemE1} and Lemmas \ref{lem1.1} we could continue from that solution a~$C^1$ branch of solutions backward up to $\lm=0$.
    Because of Lemma \ref{lem1.1}, any such curve could not meet $\mathcal{G}_\eps(\om)$ at any~$0<\lm<\lm_0(\om,p)$, while
     since there is exactly one solution of $\prl$ for $\lm=0$, it should also necessarily meet $\mathcal{G}_\eps(\om)$ at $\lm=0$, which then should be a bifurcation point.
    This is in contradiction with Lemma \ref{lem1.1} and proves the uniqueness part. Let us denote by $\mathcal{G}_*(\om)$ the unique branch of solutions of $\prl$ for $\lm\in[0,\lm_*(\om,p))$ determined in this way, then by the uniqueness we have $\mathcal{G}_0(\om)\subset\mathcal{G}_*(\om)$ unless $p=1$, as claimed.\\
    The monotonicity of $\all$ and $\el$ follows from Propositions \ref{pr-enrg} and \ref{pr3.2.best}.\\
    Next we prove the claim about the inequality $\lm_+(\om,p)\geq \lm_0(\om,p)$.
    It is well known that (see \cite{Temam1977remarks}) for $p=1$ we have $\lm_+(\om,1)=\lm_0(\om,1)=\Lambda(\om,2)$, which is just the first eigenvalue
    of the Laplace operator on $\om$ with Dirichlet boundary conditions.\\
    Thus we are left to prove that  $\lm_+(\om,p)>\lm_0(\om,p)$ as far as $p\in(1,p_{_N})$.
    For the sake of simplicity from now on we set $\lm_+=\lm_+(\om,p)$. Recalling that (see \cite{BJW2024sharp}) $\lm_+$ is positive and well defined, we argue by contradiction and assume that $\lm_+\leq \lm_0(\om,p)$. Since $\lm_*(\om,p)>\lm_0(\om,p)$, by uniqueness and Lemma \ref{lemE1} we can pass to the limit along $\mathcal{G}_*(\om)$ as $\lm\to (\lm_+)^{-}$ and come up with at least one solution $(\al_+,\psi_+)$ of $\prl$ for $\lm=\lm_+$.
    Clearly $\al_+\leq 0$ and then $u_+=\lm_+\psi_+$ would be a solution of
\begin{align}\label{us}
\graf{-\Delta u_+ = \lm_+ [\al_+ + u_+]_+^p \quad \mbox{in} \;\;\om\\ \\
\bigintss\limits_{\om}[\al_+ + u_+]_+^p=1\\ \\
u_+>0 \;\;\mbox{in}\;\;\om, \quad u_+=0 \;\; \mbox{on}\;\;\pa\om.
}
\end{align}
Now the linearization of \eqref{us} where one just disregards the integral constraint takes the form
\begin{align}\label{phis}
\graf{-\Delta \phi = \lm_+ p[\al_+ + u_+]_+^{p-1}\phi  \quad \mbox{in} \;\;\om\\ \\
\phi=0 \quad \mbox{on}\;\;\pa\om,
}
\end{align}
and it is readily seen that $u_+$ is a positive strict subsolution of \eqref{phis}, whence the first
eigenvalue of \rife{phis} is negative. In particular we infer that
$$
0>
\inf\limits_{w\in H^{1}_0(\om)\setminus\{0\}}
\dfrac{\ino |\nabla w|^2}{\ino [\al_+ + \lm_+\psi_+]_+^{p-1} w^2}-\lm_+ p=
$$
$$
\inf\limits_{w\in H^{1}_0(\om)\setminus\{0\}}
\dfrac{\ino |\nabla w|^2-\lm_+ p \ino [\al_+ + \lm_+\psi_+]_+^{p-1} w^2  }{\ino [\al_+ + \lm_+\psi_+]_+^{p-1} w^2}=
\nu_{1}(\al_+,\psi_+),
$$
which contradicts Proposition \ref{preig}.
\end{proof}

\begin{proof}[Proof of Theorem~\ref{thm:gel}]
     We recall from \cite[Appendix A]{BJW2024sharp} that $\lm_+(\om,p)<+\ii$.
     To simplify the exposition we denote $\lm_*=\lm_*(\om,p)$ and $\lm_+=\lm_+(\om,p)$ and use for the time being the temporary notation~$\lm_{+\wedge *}\equiv\min\braces{\lm_*,\lm_+}$.

     In view of Theorem \ref{thm1}, the set of solutions of $\prl$ for $\lm<\lm_{+\wedge*}$ is a $C^1$ curve $\lm\mapsto(\all,\pl)$ which we denote by ${\mathcal{G}}_{+\wedge*}$, where in particular $\frac{d \el}{d\lm}>0$, $\frac{d \all}{d\lm}<0$ along ${\mathcal{G}}_{+\wedge*}$.
     Clearly $\all\searrow 0^+$ as far as $\lm\nearrow \lm_+$ and $\lm_+<\lm_*$.
     As solutions are positive for~$\lm<\lm_+$, we have $[\al+\lm\psi]_+^p\equiv (\al+\lm\psi)^p$ which implies that the map $F$ defined in \eqref{eF} is jointly real analytic as a function of $(\lm, \al,\psi)$, as far as $\lm<\lm_+$.
     As a consequence, we can use the analytic implicit function theorem (see e.g.~\cite[Theorem 4.5.4]{BuffoniToland2003analytic})
     in the proof of Lemma \ref{lem1.1} to deduce that in fact $\lm\mapsto(\all,\pl)$ is real analytic.

     The variables~$\vla=\frac{\lm}{\all}\pl$ and $\mu=\mul=\lm\all^{p-1}$ are also real analytic, and~$\mul\to 0^+$ as $\lm\nearrow \lm_+$
     since $\all\searrow 0^+$ as $\lm\nearrow \lm_+$. Actually $\frac{d\mul}{d\lm}<0$ as $\lm\nearrow \lm_+$ and $\lm_+< \lm_*$. In fact we have
     $$
     \frac{d\mul}{d\lm}=\all^{p-2}(\all+\lm(p-1) \frac{d\all}{d\lm})\leq \all^{p-2}(o(1)-2\lm(p-1) \el),
     $$
     for $\lm$ close enough to $\lm_+$, where we used \eqref{Entropyd} in Proposition \ref{pro:lat.1} below.

      In particular $\vla$ solves $\pqm$ for $\mu=\mul$ and it is readily seen that $\el=\frac12 \ino\rl\pl$ takes the form
      $$
        \el=\frac12\left(\frac{\all}{\lm}\right)^2\mul\ino (1+\vla)^p\vla,
        $$
    and is real analytic as well.
    However because of Lemma \ref{lemE1} we see that $\el=\frac12 \ino\rl\pl$ is uniformly bounded for $\lm\leq \lm_+$, whence $\el\nearrow E_{\ii}$ for some $E_{\ii}\in (E_0,+\ii)$, as $\lm\nearrow \lm_+$.

    We are just left to prove that $\|\vla\|_{\ii}\to +\ii$ as $\lm\nearrow \lm_+$.
    Since $\vla=\frac{\lm}{\all}\pl$ and $\all\to 0^+$ as $\lm\nearrow \lm_+$, it is enough to prove that $\liminf_{\lm\to \lm_+^-}\|\pl\|_\ii>0$.
    If not there exist sequences $\lm_n\nearrow \lm_+$,~$\al_n\searrow 0^+$ and $\psi_n$ such that $\|\psi_n\|_\ii\to 0$.
    Again by Lemma \ref{lemE1} and standard elliptic estimates, possibly passing to a subsequence, $\psi_n$ converges in $C^2_0(\ov{\om})$ to a solution $\psi$ of
    \begin{equation}\label{psiplus}
      \graf{-\Delta \psi ={\lm_+^p}\psi^p\quad \mbox{in}\;\;\om\\
    \bigintss\limits_{\om} {\dsp \lm_+^p\psi^p}=1\\
    \psi=0 \quad \mbox{on}\;\;\pa\om
    }
    \end{equation}
    which is impossible since we would find that $1=\int_{\om} {\dsp \lm_+^p\psi^p}=\lim\limits_{n\to +\ii}\int_{\om} {\dsp (\al_n+\lm_n\psi_n)^p}=0$.
\end{proof}

For later references we state the following formula of independent interest.

\begin{prop}\label{pro:lat.1}

    Let $N\geq 2$, $p\in [1,p_N)$ and assume that $(\all,\pl)$ is differentiable at some $\lambda>0$.
    Then, 
    $$
        \frac{\dd}{\dd \lm} (\all +2\lm \el)=\parenthesis{1+\frac{1}{p}}\lm \frac{\dd \el}{\dd\lm}.
    $$

    In particular we have,
     \begin{align}\label{Entropyd}
        \frac{\dd\all}{\dd\lm}+2\el = - \parenthesis{1-\frac{1}{p}}\lm\frac{\dd \el}{\dd \lm}.
    \end{align}
\end{prop}

\begin{proof}
Using the notation from Propositions~\ref{pr-enrg} and~\ref{pr3.2.best}, we have that,
\begin{align}
    \el=\frac{1}{2\lm}\ino [\all+\xil]_+^p\xil=\frac{1}{2\lm}\ino \rl \xil.
\end{align}
Recall that $\wl=\frac{d\xil}{d\lm}\in C^{2,r}_0(\,\ov{\om}\,)$ for $p>1$, $\wl\in W^{2,t}_0({\om})$ for some $t>N$ for $p=1$ and satisfies~\eqref{malfa.1}. 
Then we have,
\begin{align}
    \frac{d \el}{d\lm}
    =&-\frac{\el}{\lm}+\frac{p}{2\lm}\ino \rlq [\wl]_{\ssl} \xil+\frac{1}{2\lm}\ino\rl \wl\\
=&-\frac{\el}{\lm}+\frac{p}{2\lm}\ino \rl [\wl]_{\ssl}
   +\frac{1}{2\lm}\ino\rl \wl \\
=&-\frac{\el}{\lm}+\frac{p}{2\lm}\ino \rl [\wl]_{\ssl}+\frac{1}{2\lm}\ino\rl [\wl]_{\ssl}
  +\frac{1}{2\lm}\Abracket{\wl}_{\ssl}\\
=&-\frac{\el}{\lm}+\frac{p+1}{2\lm}\ino \rl [\wl]_{\ssl}+\frac{1}{2\lm}\Abracket{\wl}_{\ssl}.
\end{align}
This equality can be reformulated as follows,
\begin{align}
\frac{d}{d \lm} (\all +2\lm \el)
=&-\Abracket{\wl}_{\ssl}+2\el+2\lm \frac{d \el}{d\lm}\\
=&({p+1})\ino \rl [\wl]_{\ssl}
  =({p+1})\ino \rlq (\all+\xil) [\wl]_{\ssl} \\
=&({p+1})\ino \rlq [\wl]_{\ssl}\xil.
\end{align}
On the other side, by testing~\eqref{wl} against~$\xil$, we see that
$$
\lm p \ino \rlq [\wl]_{\ssl}\xil
=\ino\rl (\lm \wl-\xil)
=\lm^2\ino \rl \frac{\wl-\pl}{\lm}
=\lm^2 \ino\rl \frac{d\pl}{d\lm}
=\lm^2\frac{d\el}{d\lm},
$$
where we used the first equality in \eqref{8.12.11}.
As a consequence we have,
$$
\frac{d}{d \lm} (\all +2\lm \el)
=({p+1})\ino \rlq [\wl]_{\ssl} \ul
=\frac{p+1}{p}\lm \frac{d\el}{d\lm}.
$$
\end{proof}
As a corollary of \eqref{Entropyd}, we see that the derivatives of~$\all$ and~$\el$ have to diverge simultaneously, whenever this could be the case.


\section{The solution branch on the ball of unit volume}\label{sect:calc in ball}

In the proof of Theorem \ref{thm:lm*} we will need a refined knowledge of the regularity properties of $\all$, which is why we analyze here the solution curve of~$\prl$ on~$ \mathbb{D}_N=B_{R_N}$, with unit volume~$|B_{R_N}|=1$.\\

Remark that for $\om=\mathbb{D}_2$, the uniqueness for the Grad-Shafranov type problem equivalent to $\prl$ is well-known and dates back to \cite{BandleSperb1983qualitative}, while for $N\geq 3$ has been recently proved in \cite{BJW2024sharp}.\\
For~$N\geq 2$, $p\in (1,p_{_N})$ and~$I>0$ let us define the Grad-Shafranov type  free boundary problem (see e.g. \cite{BJW2024sharp, Berestycki1980free})
$$
\graf{-\Delta {\rm v} =[{\rm v}]_+^p\quad \mbox{in}\;\;\mathbb{D}_N\vspace{0.5em} \\
\quad\;\; {\rm v}=\gamma \qquad \mbox{on}\;\;\pa\mathbb{D}_N\vspace{0.5em} \\
\bigints\limits_{\mathbb{D}_N} {\dsp [{\rm v}]_+^p}=I
}\qquad \fbi
$$
where the unknown is $(\gamma,{\rm v})\in \R \times
C^{2,\beta}(\ov{\mathbb{D}_N})$.
Let~$I=\lambda^{\frac{p}{p-1}}$, then a solution~$(\gamma, {\rm v} )$ of~$\fbi$ corresponds to a solution~$(\alpha,\psi)\in \R\times C^{2,\beta}(\overline{\mathbb{D}_N})$ of~$\prl$ where
\begin{align}\label{eq:change variable}
   \begin{cases}
    \gamma= \lambda^{\frac{1}{p-1}}\alpha, \vspace{2mm}\\
    {\rm v} = \lambda^{\frac{1}{p-1}}(\alpha+\lambda\psi).
   \end{cases}
  \end{align}
The solutions for~$\fbi$ can be given explicitly for any~$I>0$.
Indeed, let~$u_0=u_{0,N}$ be the~$N$-dimensional Lane-Emden solution in the unit ball~$B_1(0)\subset\R^N$ as defined in the introduction, namely
\begin{align}
    \graf{-\Delta u_0 =u_0^p\quad \mbox{in}\;\;B_1\\
u_0>0 \quad \mbox{in}\;\;B_1\\
u_0=0 \quad \mbox{on}\;\;\pa B_1.
}
\end{align}
As far as no ambiguity could arise, we will omit the index~$N$ and simply write~$u_0=u_{0,N}$.
By~\cite{GNN1979symmetry} we know that~$u_0$ is radial and radially decreasing.
Moreover since~$p<p_{_N}$ is subcritical, the solution is smooth.

\subsection{Planar case \texorpdfstring{$N=2$}{N=2}}

In this case~$R_2=\frac{1}{\sqrt{\pi}}$, and the solution~${\rm v}$ of~$\fbi$ is radial, of class $C^2$ and can be explicitly evaluated in terms of $\phi$ as follows:
\begin{itemize}
    \item as far as ~$\gamma={\rm v}|_{\partial \mathbb{D}_2} \geq 0$, we have,
            \begin{align}
                {\rm v}(x) = \frac{1}{R^{\frac{2}{p-1}}}u_0(\frac{x}{R}),
            \end{align}
            for some~$R\in [R_2,+\infty)$;
    \item as far as~$\gamma<0$, we have,
             \begin{align}
                {\rm v}(x)=
                \begin{cases}
                    \frac{1}{R^{\frac{2}{p-1}}}u_0(\frac{x}{R}), & 0\leq |x|\leq R, \vspace{0.5em} \\
                    A\log\frac{|x|}{R} + B, & R< |x| \leq R_{2},
                \end{cases}
            \end{align}
            for some~$R\in (0,R_2)$ and~$A,B\in\R$.
\end{itemize}

In the non negative boundary value case,~${\rm v}$ is smooth, while in the negative boundary value case, the~$C^2$ regularity requires that
\begin{align}
    A=\frac{u_0'(1)}{R^{\frac{2}{p-1}}}<0 \quad\mbox{ and }\quad B=0.
\end{align}
The value of~$R>0$ is uniquely determined by the integral constraints:
\begin{align}
    I =  \int_{\mathbb{D}_N} [{\rm v}]_+^p
    =\begin{cases}
         \int_{B_{R}} \frac{1}{R^{\frac{2p}{p-1}}} u_0(\frac{x}{R})^p
        =\frac{1}{R^{\frac{2}{p-1}}}\int_{B_1} u_0(y)^p\dd{y}, & \mbox{ if } \gamma<0, \mbox{ i.e. } R < R_2, \\
        \int_{B_{R_2}} \frac{1}{R^{\frac{2p}{p-1}}} u_0(\frac{x}{R})^p
                =\frac{1}{R^{\frac{2}{p-1}}}\int_{B_{R_2/ R}} u_0(y)^p\dd{y}, & \mbox{ if } \gamma\geq 0, \mbox{ i.e. } R\geq R_2. \vspace{0.5em}
    \end{cases}
\end{align}
To see that for each~$I$ there is a unique~$R$ satisfying the integral constraint, we note that the function
\begin{align}
    R\mapsto \mathscr{I}(R)\equiv
        \begin{cases}
            \frac{1}{R^{\frac{2}{p-1}}} \int_{B_{1}} u_0(y)^p\dd{y}, & \mbox{ if } R<R_2, \\
            \frac{1}{R^{\frac{2}{p-1}}} \int_{B_{R_2/ R}} u_0(y)^p\dd{y}, & \mbox{ if } R\geq R_2,
        \end{cases}
\end{align}
is continuous and monotonic decreasing, with
    \begin{align}
        \lim_{R\searrow 0} \mathscr{I}(R)= +\infty, & &
        \mathscr{I}(R_2)=\frac{1}{R_2^{\frac{2}{p-1}}} \int_{B_1} u_0(y)^p\dd{y} = \pi^{p-1}\int_{B_1} u_0(y)^p\dd{y}, & &
        \lim_{R\nearrow +\infty} \mathscr{I}(R)=0^+.
    \end{align}
Hence for each~$I>0$ there is a unique~$R>0$ such that~$\mathscr{I}(R)=I$ given by the inverse function~$\mathscr{I}^{-1}(I)$.

    \begin{lemma}
        The function~$\mathscr{I}(R)$ is~$C^2$ in~$R$.
    \end{lemma}
    \begin{proof}
        In the interval~$(0, R_2)$ it is smooth in~$R$.
        We need to consider the differentiability at~$R_2$ and in the interval~$(R_2, +\infty)$.

        (a) In the interval~$(R_2,+\infty)$, we can differentiate the function~$\mathscr{I}(R)$:
        \begin{align}
            \frac{\dd\mathscr{I}}{\dd R}
            =& -\frac{2}{p-1} \frac{1}{R^{\frac{2}{p-1}+1}} \int_{B_{\frac{R_2}{R}}} u_0(y)^p\dd{y}
             + \frac{1}{R^{\frac{2}{p-1}}} \int_{\partial B_{\frac{R_2}{R}}} u_0(y)^p\dd{s} \parenthesis{ -\frac{R_2}{R^2} } \\
            =& -\frac{2}{p-1}\frac{1}{R^{\frac{2}{p-1}+1}} \int_{B_{\frac{R_2}{R}}} u_0(y)^p\dd{y}
              -\frac{R_2}{R^{\frac{2}{p-1}+2}} \cdot 2\pi \frac{R_2}{R}u_0(\frac{R_2}{R})^p \\
            =&-\frac{2}{p-1}\frac{1}{R^{\frac{2}{p-1}+1}} \int_{B_{\frac{R_2}{R}}} u_0(y)^p\dd{y}
             -\frac{2\pi R_2^2}{ R^{\frac{2}{p-1}+3}} u_0(\frac{R_2}{R})^p<0
        \end{align}
        In particular, at~$R=R_2$, the second summand vanishes and we deduce that,
        \begin{align}
            \frac{\dd}{\dd R}\Big|_{R=R_2+0} \mathscr{I}(R)
            =\frac{\dd}{\dd R}\Big|_{R=R_2-0} \mathscr{I}(R),
        \end{align}
        whence~$\mathscr{I}(R)$ is~$C^1$.

        (b) In the interval~$(R_2,+\infty)$, we can take the second derivative of the function~$\mathscr{I}(R)$:
        \begin{align}
            \frac{\dd^2\mathscr{I}}{\dd R^2}
            =& \frac{2}{p-1}\parenthesis{\frac{2}{p-1}+1} \frac{1}{R^{\frac{2}{p-1}+2}} \int_{B_{\frac{R_2}{R}}} u_0(y)^p\dd{y}
             -\frac{2}{p-1}\frac{1}{R^{\frac{2}{p-1}+1}} \int_{\partial B_{\frac{R_2}{R}}} u_0(y)^p\dd{s} \cdot\parenthesis{-\frac{R_2}{R^2}} \\
             & +\parenthesis{\frac{2}{p-1}+3}\frac{2\pi R_2^2}{R^{\frac{2}{p-1}+4}} u_0(\frac{R_2}{R})^p
              -\frac{2\pi R_2^2 }{ R^{\frac{2}{p-1}+3}} \cdot pu_0(\frac{R_2}{R})^{p-1}u_0'(\frac{R_2}{R})\parenthesis{-\frac{R_2}{R^2}} \\
            =&  \frac{2}{p-1}\parenthesis{\frac{2}{p-1}+1} \frac{1}{R^{\frac{2}{p-1}+2}} \int_{B_{\frac{R_2}{R}}} u_0(y)^p\dd{y}
            +\frac{2}{p-1}\frac{2\pi R_2^2}{R^{\frac{2}{p-1}+4}}u_0(\frac{R_2}{R})^p  \\
             &+\parenthesis{\frac{2}{p-1}+3}\frac{2\pi R_2^2}{R^{\frac{2}{p-1}+4}} u_0(\frac{R_2}{R})^p
             +p\frac{2\pi R_2^3}{R^{\frac{2}{p-1}+5}} u_0(\frac{R_2}{R})^{p-1} u_0'(\frac{R_2}{R})\\
            =& \frac{2}{p-1}\parenthesis{\frac{2}{p-1}+1} \frac{1}{R^{\frac{2}{p-1}+2}} \int_{B_{\frac{R_2}{R}}} u_0(y)^p\dd{y}
             +\parenthesis{\frac{4}{p-1}+3}\frac{2\pi R_2^2}{R^{\frac{2}{p-1}+4}} u_0(\frac{R_2}{R})^p \\
             &+p\frac{2\pi R_2^3}{R^{\frac{2}{p-1}+5}} u_0(\frac{R_2}{R})^{p-1} u_0'(\frac{R_2}{R}).
        \end{align}
        Again, since~$u_0(1)=0$ and~$p>1$, we see that,
        \begin{align}
            \frac{\dd^2}{\dd R^2}\Big|_{R=R_2+0} \mathscr{I}(R)
            =\frac{\dd^2}{\dd R^2}\Big|_{R=R_2-0} \mathscr{I}(R),
        \end{align}
        whence~$\mathscr{I}(R)$ is~$C^2$.
    \end{proof}

    To come up with higher differentiability of~$\mathscr{I}(R)$, we would need that~$p-1>1$, otherwise we miss the differentiability at~$R=R_2$. However, for~$R>R_2$, since~$u_0$ is smooth in~$B_1(0)$, we see that~$\mathscr{I}(R)$ is smooth for~$R>R_2$.

    \

    \begin{cor}
        The function~$R=\mathscr{I}^{-1}(I)$ is~$C^2$.
    \end{cor}
    \begin{proof}
        It suffices to note that~$\frac{\dd \mathscr{I}}{\dd R}$ never vanishes, hence~$\mathscr{I}^{-1}$ is as regular as~$\mathscr{I}$ does, whence it is of class~$C^2$.
    \end{proof}

\subsection{The higher dimensional case \texorpdfstring{$N\geq 3$}{N>=3}}

    The same argument works for general dimensions~$N\geq 3$, but for form of the harmonic part which has to be correspondingly modified. We state it here for the sake of completeness.

    Recall that the volume of the unit ball~$B_1(0)\subset\R^N$ is~$|B_1|=\omega_N$ and the volume of the unit sphere is~$|\pa B_1|=N\omega_N$.
    The value of~$R_N>0$ is fixed so that~$|B_{R_N}|=\omega_N R_N^N=1$.

    The unique solution~${\rm v}$, which has to be radial and radially decreasing, takes the form
    \begin{itemize}
        \item for~$\gamma\geq 0$:
                \begin{align}
                    {\rm v}(x)=\frac{1}{R^{\frac{2}{p-1}}} u_0(\frac{x}{R})
                \end{align}
                for some~$R\in [R_N,+\infty)$;
        \item for~$\gamma<0$:
                \begin{align}
                    {\rm v}(x) =
                    \begin{cases}
                        \frac{1}{R^{\frac{2}{p-1}}} u_0(\frac{x}{R}), & 0\leq |x|\leq R, \vspace{0.2em} \\
                        A\parenthesis{\frac{1}{|x|^{N-2}} -\frac{1}{R^{N-2}}}+B, & R\leq |x| \leq R_N,
                    \end{cases}
                \end{align}
                for some~$R\in(0,R_N)$ and~$A,B\in\R$.
                The~$C^2$ regularity of~${\rm v}$ requires that,
                \begin{align}
                    A=\frac{-u_0'(1)}{N-2}\frac{1}{R^{\frac{2}{p-1} -N+2 }}>0 \quad \mbox{ and } \quad B=0.
                \end{align}
    \end{itemize}
    Again the parameter~$R>0$ will be determined uniquely by the integral constraints as follows,
    \begin{align}
        R\mapsto \mathscr{I}(R) =
        \begin{cases}
            \frac{1}{R^{\frac{2}{p-1} -N+2 } } \int_{B_1} u_0(y)^p\dd{y}, & \mbox{ if } R < R_N, \vspace{0.4em }\\
            \frac{1}{R^{\frac{2}{p-1} -N+2} } \int_{B_{\frac{R_N}{R}}} u_0(y)^p\dd{y}, & \mbox{ if } R\geq R_N.
        \end{cases}
    \end{align}
    \begin{lemma}
        The function~$\mathscr{I}(R)$ is~$C^2$ and monotonic strictly decreasing, with the inverse function~$\mathscr{I}^{-1}$ also of class~$C^2$.
    \end{lemma}
    \begin{proof}
        Note that
        \begin{align}
            \frac{2}{p-1}-N+2= \frac{N-2}{p-1}\parenthesis{p_{_N} -p}>0, \quad \mbox{ since } p< p_{_N},
        \end{align}
        whence~$\mathscr{I}(R)$ is monotonic decreasing.
        Also the function~$\mathscr{I}$ is readily seen to be smooth in~$(0,R_N)$ and in~$(R_N,+\infty)$ and
        we are just left to verify that it is twice differentiable at~$R_N$.
        This can be done as follows.
        First consider the derivative in~$(R_N,+\infty)$:
        \begin{align}
            \frac{\dd\mathscr{I}}{\dd R}
                =& - \frac{N-2}{p-1}(p_{_N}-p) \frac{1}{R^{\frac{2}{p-1}-N+3}}\int_{B_{\frac{R_N}{R}}} u_0(y)^p\dd{y}
                +\frac{1}{R^{\frac{2}{p-1} -N+2}} \int_{\partial B_{\frac{R_N}{R}}} u_0(y)^p\dd{s} \cdot \parenthesis{-\frac{R_N}{R^2}} \\
                =&- \frac{N-2}{p-1}(p_{_N}-p) \frac{1}{R^{\frac{2}{p-1}-N+3}}\int_{B_{\frac{R_N}{R}}} u_0(y)^p\dd{y}
                -N\omega_N\frac{ R_N^N}{R^{\frac{2}{p-1}+3} } u_0(\frac{R_N}{R})^{p},
        \end{align}
         which is strictly negative, and observe that at~$R=R_N$ the second term vanishes in view of~$u_0(1)=0$.
            We thus see that
            \begin{align}
                \frac{\dd}{\dd R}\Big|_{R=R_N+0} \mathscr{I}(R)
                = \frac{\dd}{\dd R}\Big|_{R=R_N-0} \mathscr{I}(R)
            \end{align}
            and hence~$\mathscr{I}(R)$ is a~$C^1$ function in~$(0,+\infty)$.
            In particular, the derivative never vanishes, so its inverse function~$R_F(I)$ is also~$C^1$.

            Then consider the second derivative of~$\mathscr{I}(R)$ in~$(R_N,+\infty)$:
            \begin{align}
                \frac{\dd^2\mathscr{I}}{\dd R^2}
                =& \left[\frac{N-2}{p-1}(p_{_N}-p)\right] \left[\frac{N-2}{p-1}(p_{_N}-p) + 1\right] \frac{1}{R^{\frac{2}{p-1}-N+4}} \int_{B_{\frac{R_N}{R}}} u_0(y)^p\dd{y}  \\
                 & + \frac{N-2}{p-1}(p_{_N}-p)N\omega_N \frac{ R_N^N }{ R^{\frac{2}{p-1}+4}} u_0(\frac{R_N}{R})^p \\
                 &+ \parenthesis{\frac{2}{p-1}+3} N\omega_N\frac{ R_N^N }{R^{\frac{2}{p-1}+4} } u_0(\frac{R_N}{R})^p \\
                 &+ p N\omega_N \frac{ {R_N}^{N+1} }{ R^{\frac{2}{p-1}+5}} u_0(\frac{R_N}{R})^{p-1} u_0'(\frac{R_N}{R})  \\
             =&\left[\frac{N-2}{p-1}(p_{_N}-p)\right] \left[\frac{N-2}{p-1}(p_{_N}-p) + 1\right] \frac{1}{R^{\frac{2}{p-1}-N+4}} \int_{B_{\frac{R_N}{R}}} u_0(y)^p\dd{y}  \\
                 & + \parenthesis{\frac{4}{p-1}-N+5}N\omega_N \frac{ R_N^N }{ R^{\frac{2}{p-1}+4}} u_0(\frac{R_N}{R})^p
                 + p N\omega_N \frac{ {R_N}^{N+1} }{ R^{\frac{2}{p-1}+5}} u_0(\frac{R_N}{R})^{p-1} u_0'(\frac{R_N}{R})
            \end{align}
            We see that~$\mathscr{I}(R)$ is~$C^2$ for~$R>R_N$.
            Moreover, since~$p>1$ and~$u_0(1)=0$, we have
            \begin{align}
                 \frac{\dd^2}{\dd R^2}\Big|_{R=R_N+0} \mathscr{I}(R)
                = \frac{\dd^2}{\dd R^2}\Big|_{R=R_N-0} \mathscr{I}(R)
            \end{align}
            whence~$\mathscr{I}\in C^2(0,+\infty)$.
            Consequently the inverse function~$\mathscr{I}^{-1}(I)$ is also~$C^2$ in~$(0,+\infty)$.
    \end{proof}

\subsection{Differentiability of \texorpdfstring{$\all$}{alpha-lambda}}

    By~\eqref{eq:change variable} and the equivalence between~$\fbi$ and~$\prl$, we come up with a refined description of the unique solution on~$\mathbb{D}_N$.
    In particular, for each~$\lambda>0$ there exists a unique solution~$(\all,\pl)$ and our main concern is the differentiability of~$\all$ in~$\lm\in (0,+\infty)$.

    The boundary value of~${\rm v}$ is
    \begin{align}
        \gamma={\rm v}(R_N)=
        \begin{cases}
            \frac{1}{R^{\frac{2}{p-1}}}u_0(\frac{R_N}{R})\geq 0, &  \mbox{ if } R\geq R_N, \vspace{0.5em} \\
            \frac{u_0'(1)}{R^{\frac{2}{p-1}}}\log\frac{R_2}{R}<0,  & \mbox{ if } N=2, \; R<R_2, \vspace{0.5em } \\
            \frac{u_0'(1)}{N-2}\frac{1}{R^{\frac{2}{p-1}}}\parenthesis{1- (\frac{R}{R_N})^{N-2} }<0, & \mbox{ if } N\geq 3, \; R<R_N.
        \end{cases}
    \end{align}
    Thus~$\gamma(R)$ is a~$C^2$ function in~$R\in(0,+\infty)$ and it would be interesting to better understand $\gamma(R)$, a task that we will pursue elsewhere.

   Thus, given~$\lambda>0$, we have~$I=\lambda^{\frac{p}{p-1}}=\lambda^{1+\frac{1}{p-1}}$ and~$R=\mathscr{I}^{-1}(I)=\mathscr{I}^{-1}(\lambda^{1+\frac{1}{p-1}})$, hence
    \begin{align}
        \all=\frac{\gamma}{\lm^{\frac{1}{p-1}}}
        =\frac{1}{\lambda^{\frac{1}{p-1}}}\gamma\parenthesis{ \mathscr{I}^{-1}(\lambda^{1+\frac{1}{p-1}}) }.
    \end{align}
    As a composition of~$C^2$ functions, we conclude that~$\all$ is~$C^2$ in~$\lm\in (0,+\infty)$.

    \

    For later convenience, we will denote
    \begin{align}
         R(\lambda)\coloneqq \mathscr{I}(\lambda^{\frac{p}{p-1}})
    \end{align}
    which is the value of~$R$ corresponding to~$I=\lambda^{\frac{p}{p-1}}>0$.
    Also we denote
    \begin{align}
        I_p=\int_{B_1} u_0(y)^p \dd{y}
    \end{align}
    which should not be confused with the integral constraint value~$I>0$.

    Collecting the above facts, together with some elementary calculations leads to the following

    \begin{prop}
        Let~$N=2$ and~$(\all,\pl)$ be the unique solution of {\rm $\prl$} on~$\mathbb{D}_2$.
        \begin{itemize}
            \item[(i)] $\lm_+(\mathbb{D}_2,p)=\pi^{\frac{1}{p}} I_p^{\frac{p-1}{p}}$.
            \item[(ii)] If~$\lm\in[0,\lm_+(\mathbb{D}_2,p)]$, then~$R(\lm)\geq R_2$ and
                        \begin{align}\label{eq:gamma-alpha-lambda-2}
                            \gamma_{\ssl}=\frac{1}{R(\lm)^{\frac{2}{p-1}}}u_0(\frac{R_2}{R(\lm)}), & &
                            \all=\frac{\gamma_{\ssl}}{\lm^{\frac{p}{p-1}}}
                            =\frac{1}{\lm^{\frac{1}{p-1}} R(\lm)^{\frac{2}{p-1}}}u_0(\frac{R_2}{R(\lm)}).
                        \end{align}
            \item[(iii)] If~$\lm\in (\lm_+(\mathbb{D}_2,p), +\infty)$, then
                        \begin{align}
                            R(\lm)=I_p^{\frac{p-1}{2}} \lm^{-\frac{p}{2}} < R_2,
                        \end{align}
                        \begin{align}
                            \gamma_{\ssl}
                            =&\frac{p-1}{2}u_0'(1)\frac{\lambda^{\frac{p}{p-1}}}{I_p}\log\parenthesis{\frac{R_2^{\frac{2}{p-1}} \lm^{\frac{p}{p-1}}}{I_p}}<0,
                            \\
                            \all
                            =&\frac{\gamma_{\ssl}}{\lm^{\frac{1}{p-1}}}
                            =\frac{p-1}{2}u_0'(1)\frac{\lambda}{I_p} \log\parenthesis{\frac{R_2^{\frac{2}{p-1}} \lm^{\frac{p}{p-1}}}{I_p}}<0.
                        \end{align}
            \item[(iv)] $\all$ is~$C^2$ in~$\lm\in (0,+\infty)$.
        \end{itemize}
    \end{prop}

    \begin{prop}
        For~$N\geq 3$, let~$(\all,\pl)$ be the unique solution of~{\rm $\prl$} on~$\mathbb{D}_N$.
        \begin{itemize}
            \item[(i)] $\lm_+(\mathbb{D}_N,p)$ is uniquely determined by
                        \begin{align}
                           \lm_+(\mathbb{D}_N,p)^{\frac{p}{p-1}}= \frac{I_p}{R_N^{\frac{2}{p-1} -N+2 } }.
                        \end{align}
            \item[(ii)] If~$\lm\in[0,\lm_+(\mathbb{D}_N,p)]$, then~$R(\lm)\geq R_N$ and
                        \begin{align}\label{eq:gamma-alpha-lambda-N}
                            \gamma_{\ssl}=\frac{1}{R(\lm)^{\frac{2}{p-1}}}u_0(\frac{R_N}{R(\lm)}), & &
                            \all=\frac{\gamma_{\ssl}}{\lm^{\frac{p}{p-1}}}
                            =\frac{1}{\lm^{\frac{1}{p-1}} R(\lm)^{\frac{2}{p-1}}}u_0(\frac{R_N}{R(\lm)}).
                        \end{align}

            \item[(iii)] If~$\lm\in(\lm_+(\mathbb{D}_N,p),+\infty)$, then
                        \begin{align}
                            R(\lm)=\parenthesis{ I_p \lm^{\frac{p}{p-1}}}^{\frac{1}{ \frac{2}{p-1}-N+2 }} <R_N,
                        \end{align}
                        \begin{align}
                            \gamma_{\ssl}
                            =&\frac{-u_0'(1)}{N-2}\frac{\lm^{\frac{p}{p-1}}}{I_p} \parenthesis{\frac{1}{R_N^{N-2}} -\parenthesis{\frac{\lm^{\frac{p}{p-1}}}{I_p}}^{\frac{N-2}{\frac{2}{p-1}-N+2}}  }
                            <0, \\
                            \all
                            =&\frac{-u_0'(1)}{N-2}\frac{\lm}{I_p} \parenthesis{\frac{1}{R_N^{N-2}} -\parenthesis{ \frac{\lm^{p}}{I_p^{p-1}} }^{\frac{1}{p_{_N}-p}}  }
                            <0.
                        \end{align}
            \item[(iv)] $\all$ is~$C^2$ in~$\lm\in(0,+\infty)$.
        \end{itemize}

    \end{prop}

    In the proof of Theorem \ref{thm:lm*} we will use the regularity properties of $\all$ for $\lm\in[0,\lm_+(\mathbb{D}_N,p)]$ which in turn will imply its monotonicity.
    Meanwhile we prove the monotonicity in~$(\lm_+(\mathbb{D}_N,p),+\infty)$ which is of independent interest.
    \begin{prop}\label{prop:monotonicity of alpha-lambda I}
        For~$N\geq 2$ let~$(\all,\pl)$ be the unique solution of {\rm $\prl$} on~$\mathbb{D}_N$. Then for any~$\lm\in (\lm_+(\mathbb{D}_N,p),+\infty)$,
        \begin{align}
            \frac{\dd \all}{\dd\lm}<0.
        \end{align}
    \end{prop}
    \begin{proof}
        This can be verified by direct calculation.
        Indeed, for~$N=2$,
        \begin{align}
            \frac{\dd\all}{\dd\lm}
            =& \frac{p-1}{2}\frac{u_0'(1)}{I_p} \frac{\dd}{\dd\lm}\parenthesis{\lm\log\parenthesis{\frac{R_2^{\frac{2}{p-1}} \lm^{\frac{p}{p-1}}}{I_p}} } \\
            =&\frac{p-1}{2}\frac{u_0'(1)}{I_p} \parenthesis{ \log \parenthesis{\frac{R_2^{\frac{2}{p-1}} \lm^{\frac{p}{p-1}}}{I_p}} +\lm\cdot \frac{p}{p-1}\frac{1}{\lm} } \\
            =&\frac{p-1}{2}\frac{u_0'(1)}{I_p} \parenthesis{\frac{p}{p-1} \log\frac{\lm}{\lm_+(\mathbb{D}_2,p)} +\frac{p}{p-1} } \\
            =& \frac{p}{2} \frac{u_0'(1)}{I_p} \parenthesis{ \log\frac{\lm}{\lm_+(\mathbb{D}_2,p)} +1} <0
        \end{align}
        for~$\lm>\lm_+(\mathbb{D}_2,p)$.
        Meanwhile for~$N\geq 3$,
        \begin{align}
            \frac{\dd\all}{\dd\lm}
            =&\frac{-u_0'(1)}{N-2} \frac{1}{I_p} \frac{\dd}{\dd\lm}
            \parenthesis{ \frac{\lm}{R_N^{N-2}}  -  \frac{\lm^{\frac{p}{p_{_N}-p} +1}}{I_p^{ \frac{p-1}{p_{_N}-p} }} } \\
            =&\frac{-u_0'(1)}{(N-2)I_p} \frac{\dd}{\dd\lm}\parenthesis{ \frac{\lm}{R_N^{N-2}} -\frac{\lm^{\frac{p_{_N}}{p_{_N}-p}}}{ I_p^{\frac{p-1}{p_{_N}-p}} } } \\
            =&\frac{-u_0'(1)}{(N-2)I_p} \parenthesis{\frac{1}{R_N^{N-2}} -\frac{p_{_N}}{p_{_N}-p} \frac{\lm^{\frac{p}{p_{_N}-p}}}{I_p^{\frac{p-1}{p_{_N}-p}}}  } \\
            =& \frac{-u_0'(1)}{(N-2)I_p} \frac{p_{_N}}{p_{_N}-p} \frac{1}{I_p^{\frac{p-1}{p_{_N}-p}}} \parenthesis{\frac{p_{_N} }{p}  \parenthesis{\frac{I_p}{R_N^{\frac{2}{p-1}-N+2}}}^{\frac{N-2}{ \frac{2}{p-1}-N+2}}  -(\lm^{\frac{p}{p-1}})^{\frac{N-2}{ \frac{2}{p-1}-N+2 }} } <0
        \end{align}
        for~$\lm>\lm_+(\mathbb{D}_N,p)$.
    \end{proof}


\section{The proof of Theorem \ref{thm:lm*}: $\lm_*(\mathbb{D}_N,p)\geq \lm_+(\mathbb{D}_N,p)$.
}
\label{sec:lm*}

Recall that Section~\ref{sec3.4} mainly deals with the case where~$L_{\ssl}$ is non-degenerate, namely~$0\notin\Spect(L_{\ssl})$.
In this section we need to study the case that some of the eigenvalues of~$L_{\ssl}$ is zero.
Motivated by the analysis in~\cite{BJ2022uniqueness}, we extend the argument to general dimensions.\\

For later purposes we introduce the weighted product,
\begin{align}
    \Abracket{\phi,\psi}_{\ssl}\coloneqq \frac{ \ino \rlq \phi\psi}{ \ino \rlq }, \qquad \forall \phi,\psi\in L^2(\Omega).
\end{align}
Note that~$\rlq$ is continuous, hence this weighted inner product is well-defined on~$L^2(\om)$.

\

The main ingredient is the following bending lemma of Crandall-Rabinowitz type,
\begin{prop}[{\cite[Prop. 2.7.]{BJ2022uniqueness}}] \label{prop:tranverse}
    Let~$(\all,\pl)$ be a positive solution of {\rm$\prl$} with $\lm>0$ and~$0\in \Spect(L_{\ssl})$, say~$\sg_k(\all,\pl)=0$.
    Suppose that~$\sg_k(\all,\pl)=0$ is simple with the (normalized) eigenfunction~$\phi_k\in C^{2,r}_0(\ov{\om}\,)$ satisfying~$\Abracket{\phi_k}_{\ssl}\neq 0$.

    Then there exists $\eps>0$, an open neighborhood $\mathcal{U}$ of $(\lm,\all,\pl)$ in $(0,+\ii)\times(0,1)\times C^{2,r}(\ov{\om})$ and a real analytic curve
    \begin{align}
        (-\eps,\eps) \to \mathcal{U}, \qquad s\mapsto (\lm(s), \al(s),\psi(s)),
    \end{align}
    such that
    \begin{itemize}
        \item $(\lm(0), \al(0),\psi(0))=(\lm,\all,\pl)$,
        \item $F^{-1}(0,0)\cap\mathcal{U}= \braces{(\lm(s), \al(s),\psi(s)) \mid -\eps < s< \eps }$.
    \end{itemize}
    Furthermore, locally near the given solution~$(\all,\pl)$, we have~$\psi(s)=\pl+s\phi_k+\vxi(s)$, with
                \begin{align}\label{2907.0}
                    \Abracket{ [ \phi_k]_{\sscp \lm(s)},\vxi(s) }_{\sscp \lm(s)}=0,\quad s\in (-\eps,\eps)
                \end{align}
                and~$\xi(0)=0$, while
                \begin{align}\label{2907.1}
                    \vxi^{'}(0)& \equiv 0,\\
                    \al^{'}(0)&=-\lm <\phi_k>_{\ssl},\\
                    \lm^{'}(0)&=0,\\
                    \psi^{'}(0)&= \phi_k.
                \end{align}
                Moreover,
                \begin{itemize}
                    \item either $\lm(s)=\lm$ is constant in $(-\eps,\eps)$,
                    \item or $\lm^{'}(s)\neq 0$, $\sg_k(s)\neq 0$ in $(-\eps,\eps)\setminus\{0\}$, in which case $\sg_k(s)$ is simple in $(-\eps,\eps)$ and
                \begin{align}\label{2907.5}
                    \Abracket{[\phi_k]_{\ssl},\pl}_{\ssl}\neq 0\mbox{ and } \Abracket{[\phi_k]_{\ssl},\pl}_{\ssl} \mbox{   has the same sign as } \Abracket{\phi_k}_{\ssl},
                \end{align}
                \begin{align}\label{2907.10}
                    \dfrac{\sg_k(s)}{\lm^{'}(s)}
                    =\dfrac{p<[\phi_k]_{\ssl},\psi_{\ssl}>_{\ssl}+\mbox{\rm o}(1)} {<[\phi_k]_{\ssl}^2>_{\ssl}+\mbox{\rm o}(1)},\mbox{ as }s\to 0.
                \end{align}
                \end{itemize}

\end{prop}

The Lemma is proved in~\cite{BJ2022uniqueness} for the planar case, but the argument extends exactly as it stands to higher dimension, which is why we will not repeat the proof here.

The assumption that the zero eigenvalue is simple and the corresponding eigenfunction has nonzero weighted average seems nontrivial.
But on the ball it is automatically guaranteed, by the following
\begin{prop}\label{prop:simple}
    Let $\om=B_R(0)\subset \R^N$, $N\geq 2$ and let $(\all,\pl)$ be a positive solution of {\rm$\prl$} with~$\lm>0$ and $p\in (1,p_N)$.
Suppose that $\sg_k=\sg_k(\all,\pl)=0$ and let $\phi_k$ be any
corresponding eigenfunction. Then:
\begin{itemize}
    \item[(i)] $\Abracket{\phi_k}_{\ssl}\neq 0$;
    \item[(ii)] $\sg_k(\all,\pl)$ is simple, that is, it admits at most one linearly independent eigenfunction.
\end{itemize}
\end{prop}
\begin{proof}[Proof of Proposition~\ref{prop:simple}]
    (i) Argue by contradiction and assume that~$\Abracket{\phi_k}_{\ssl}=0$.
    Then~$\phi_k$ is a classical solution of
    \begin{align}
        \begin{cases}
            -\Delta\phi_k =\tl \rlq \phi_k, & \mbox{ in } B_R(0),  \vspace{0.5em} \\
            \phi_k=0, & \mbox{ on } \partial B_R(0)
        \end{cases}
    \end{align}
    By~\cite[Propisition 3.3]{LinNi-1988counterexample},~$\phi_k$ is radial, hence
    \begin{align}
        \phi_k'(R)=\frac{\partial \phi_k}{\partial \nu}|_{\partial B_R(0)} \neq 0
    \end{align}
    which contradicts the vanishing of~$\Abracket{\phi_k}_{\ssl}$ since
    \begin{align}
        -\int_{\partial B_R(0)} \frac{\partial\phi_k}{\partial{\nu}}\dd{s}
        =\ino (-\Delta\phi_k)
        =\ino \tl\rlq\phi_k
        =\ml\Abracket{\phi_k}_{\ssl}.
    \end{align}

    (ii) Since (i) holds for any eigenfunction of~$\sg_k=0$, such an eigenvalue is necessarily simple: if there were two linearly independent eigenfunctions~$\phi_{k,i}$,~$i=1,2$, with~$\Abracket{\phi_{k,i}}_{\ssl} \neq 0$,~$i=1,2$, then we would have,
    \begin{align}
        \phi_{k,1}-\frac{\Abracket{\phi_{k,1}}_{\ssl}}{\Abracket{\phi_{k,2}}_{\ssl} } \phi_{k,2}
    \end{align}
    which is a nonzero eigenfunction of~$\sg_k$ with vanishing weighted average, a contradiction.
\end{proof}

The remaining part of this section is devoted to the study of the solutions on the ball~$B_{R_N}(0)\equiv \mathbb{D}_N$ which has unit volume.
We will use time to time when needed the following facts.

\begin{rmk}\label{rem:var}
\begin{itemize}
    \item[(i)] As far as $\lm<\lm_+(\mathbb{D}_{N},p)$, any solution is positive according to Theorems B and C.
    Consequently, the first eigenvalue $\sg_1(\all,\pl)$ as defined in Section \ref{sec3} is readily seen to coincide with the first eigenvalue defined for positive solutions directly via \eqref{4.1}, see also \cite{BJ2022uniqueness}.
    \item[(ii)]  As far as $\lm<\lm_+(\mathbb{D}_{N},p)$ we have $[\al+\lm\psi]_+^p\equiv (\al+\lm\psi)^p$ and the map $F$ defined in \eqref{eF} is jointly real analytic as a function of $(\lm, \al,\psi)$.
    Thus, whenever $\lm<\min\braces{\lm_*(\mathbb{D}_{_N},p),  \lm_+(\mathbb{D}_{_N},p)}$, we can use the analytic implicit function theorem (see Theorem 4.5.4 in \cite{BuffoniToland2003analytic}) in the proof of Lemma \ref{lem1.1} to deduce that in fact $\lm\mapsto(\all,\pl)$ is real analytic.
    As far as solutions $(\all,\pl)$ are positive for some $\lm<\lm_*(\mathbb{D}_{_N},p)$ we will always use this improved version of Lemma \ref{lem1.1} when needed with no further comments.
    \item[(iii)] By the result in \cite{Berestycki1980free} there exists at least one variational solution of {\rm $\prl$} for any $\lm$.
    Thus, in view of the uniqueness of solutions (see \cite{BandleSperb1983qualitative} for $N=2$ and Theorem C above for $N\geq 3$), any solution of {\rm $\prl$} in~$\mathbb{D}_{N}$ is a variational solution.
    For positive variational solutions, we know that~$\sg_1(\all,\pl)\geq0$, see e.g.~\cite[Lemma 2.6]{BJ2022uniqueness}, where the two-dimensional case was handled, although the argument works exactly as it stands in all dimensions.
\end{itemize}

\end{rmk}

\begin{proof}[Proof of Theorem~\ref{thm:lm*}]

Recall that~${\mathcal{G}}_{*}$ is the set of solutions for~$\lm\in [0,\lm_*(\mathbb{D}_{_N},p))$, forming a~$C^1$ curve, along which
\begin{align}
    \sg_1(\all,\pl)>0, & & \frac{\dd \el}{\dd\lm}>0, & & \frac{\dd\all}{\dd\lm}<0.
\end{align}

\

\noi\textbf{Argue by contradiction and suppose that}~$\lm_*(\mathbb{D}_{N},p)<\lm_+(\mathbb{D}_{N},p)$. \\
 According to Remark~\ref{rem:var}
the curve~$\mathcal{G}_*$ is analytic.
By Lemma \ref{lemE1}, along a subsequence we can pass to the limit as $\lm_n\to\lm_*(\mathbb{D}_{N},p)$ and deduce that $(\al_n,\psi_n)$ converges in $C^2$ to a solution~$(\al_*,\psi_*) $ of $\prl$ for $\lm=\lm_*$ with $\al=\al_*>0$.
Again by Remark \ref{rem:var} we see that~$(\al_*,\psi_*)$ is a variational solution and hence $\sg_1(\al_*,\psi_*)\geq 0$.
Then necessarily $\sg_1(\al_*,\psi_*)=0$, as otherwise by Lemma \ref{lem1.1} and the continuity of eigenvalues we would have a contradiction to the definition of $\lm_*$.
As a consequence, by Proposition~\ref{prop:simple}, we see that the transversality condition needed to apply Proposition \ref{prop:tranverse} is satisfied, whence we can continue $\mathcal{G}_{*}$ to a real analytic
parametrization without bifurcation points,
$$
\mathcal{G}_{\lm_*+\eps}=\left\{[-s_*,\eps)\ni s\mapsto (\lm(s),\al(s),\psi(s))\right\},
$$
for some $s_*>0$,
where, for some $\eps>0$, we have that for any
$s\in [-s_*,\eps)$, $(\al(s),\psi(s))$ is a positive solution of $\prl$ with $\lm=\lm(s)$ and 
$(\lm(s),\al(s),\psi(s))\in \mathcal{G}_{*}$ for $s\leq 0$.

We claim that \emph{ $\lm(s)$ is monotonic increasing along the branch}, i.e., the curve $(\lm(s),\al(s),\psi(s))$ bends to the right of $\lm_*$.
Indeed, we recall from \rife{2907.5} in Proposition \ref{prop:tranverse} that,
\begin{align}
    \Abracket{[\phi_1]_{\sscp \lm_*},\psi_*}_{\sscp \lm_*}\neq 0
\mbox{ and } \Abracket{ [\phi_1]_{\sscp \lm_*},\psi_*}_{\sscp \lm_*} \mbox{  has the same sign as }
<\phi_1>_{\sscp \lm_*}.
\end{align}

At this point we use \rife{2907.10}
in Proposition \ref{prop:tranverse}, that is, putting $\sg_1(s)=\sg_1(\al(s),\psi(s))$,
\begin{align}
    \dfrac{\sg_1(s)}{\lm^{'}(s)}=\dfrac{p\Abracket{[\phi_1]_{\sscp \lm_*},\psi_*}_{\sscp \lm_*}+\mbox{\rm o}(1)}
{\Abracket{[\phi_1]_{\sscp \lm_*}^2}_{\sscp \lm_*}+\mbox{\rm o}(1)},\mbox{ as }s\to 0.
\end{align}

For $s$ small and negative we have $\sg_1(s)>0$ and $\lm^{'}(s)>0$ implying that $\Abracket{[\phi_1]_{\sscp \lm_*},\psi_*}_{\sscp \lm_*}>0$. Therefore
we infer that $\sg_1(s)\neq 0$ and has the same sign of $\lm^{'}(s)$ for $s$ small enough in a deleted neighborhood of $s=0$.
Indeed the eigenvalues $\sg_k(s)=\sg_k(\al(s), \psi(s))$ of $L_{\ssl(s)}$ are real analytic functions of $s$ (\cite{BuffoniToland2003analytic}) and then the level sets of a fixed $\sg_k(s)$ cannot have accumulation points unless $\sg_k(s)$ is constant on $(-s_*,\eps)$.
However we can rule out the latter case since for $\lm< \lm_*$ we have $\sg_1(\all,\pl)>0$ and then
no $\sg_k(s)$ can vanish identically.

At this point we use again that any solution of $\prl$ in the ball~$B_{R_{N}}(0)$ is also a variational solution (Remark \ref{rem:var}).
Therefore, $(\al(s),\psi(s))$ is a positive variational solution and we infer that $\sg_1(s)\geq 0$ for
any $s>0$ small enough.
Since $\sg_1(s)\neq 0$ for $s\neq 0$ small enough we deduce that $\sg_1(s)>0$ for $s>0$ small and then in particular that $\lm^{'}(s)>0$. Therefore,
$\lm^{'}(s)>0$ for $s$ in a small enough right neighborhood of $s=0$, showing that the curve $(\lm(s),\al(s),\psi(s))$ bends to the right of $\lm_*$.

In particular,
\begin{align}\label{flexen}
    \lim\limits_{\lm\to \lm_*^+}\frac{d \el}{d\lm}
    =& \lim\limits_{\lm\to \lm_*^+}p\,\ml \left(\Abracket{[\pl]^2_{\ssl}}_{\ssl}+\lm\Abracket{[\pl]_{\ssl},\eta_{\ssl}}_{\ssl}\right) \\
    =& \mbox{O}(1)+p\,m_{\lm_*}\lim\limits_{\lm\to \lm_*^+}\lm\Abracket{[\pl]_{\ssl},\eta_{\ssl}}_{\ssl} \\
    =& \mbox{O}(1)+p\,m_{\lm_*} \lim\limits_{s\to \lm_*^+}\frac{1}{\lm^{'}(s)}\lm(s)\Abracket{[\psi(s)]_{\sscp \lm(s)},\psi^{'}(s)}_{\sscp \lm(s)} \\
    =&\mbox{O}(1)+p\,m_{\lm_*}\lm_*\lim\limits_{s\to \lm_*^+}\frac{1}{\lm^{'}(s)} \left(\Abracket{[\psi_*]_{\sscp \lm_*},\psi^{'}(\lm_*)}_{\sscp \lm_*}+\mbox{o}(1)\right) \\
    =&\mbox{O}(1)+p\,m_{\lm_*}\lm_* \lim\limits_{s\to \lm_*^+}\frac{1}{\lm^{'}(s)} \left(\Abracket{[\psi_*]_{\lm_*},\phi_1}_{\lm_*}+\mbox{o}(1)\right)=+\infty,
\end{align}
and similarly, with the notations of Proposition \ref{pr3.2.best},
\begin{align}\label{flexal}
    \lim\limits_{\lm\to \lm_*^+}\frac{d \all}{d\lm}
    =& -\lim\limits_{\lm\to \lm_*^+}\left(\Abracket{\pl}_{\ssl}+\lm\Abracket{\eta_{\ssl}}_{\ssl}\right) \\
    =& \mbox{O}(1)-\lim\limits_{\lm\to \lm_*^+}\lm\Abracket{\eta_{\ssl}}_{\ssl} \\
    =& \mbox{O}(1)-\lm_*\lim\limits_{s\to \lm_*^+}\frac{1}{\lm^{'}(s)}\Abracket{\psi^{'}(s)}_{\sscp \lm(s)} \\
    =& \mbox{O}(1)-\lm_*\lim\limits_{s\to \lm_*^+}\frac{1}{\lm^{'}(s)}\left(\Abracket{\phi_1}_{\lm_*}+\mbox{o}(1)\right)=-\ii.
\end{align}
This is impossible, as we have already seen that~$\all$ is a global~$C^2$ function of~$\lm\in(0,+\infty)$. Therefore we conclude that~$\lm_*(\mathbb{D}_N,p)\geq \lm_+(\mathbb{D}_N,p)$.

We are left with showing that $\all\searrow 0^+$ as $\lm\nearrow \lm_+$, since then $\mul\searrow 0^+$ and $\|v^{(\mu)}\|_\ii\to +\ii$ follow immediately as in
Theorem \ref{thm:gel}. If $\lm_*>\lm_+$ the conclusion follows from Theorem \ref{thm:gel}, whence we can assume $\lm_*=\lm_+$. By Lemma \ref{lemE1} possibly along a subsequence we can assume that
     $\pl\to \psi_*$, so that $(\al_*,\psi_*)$ solves $\prl$ for $\lm=\lm_*=\lm_+$. By contradiction assume that $\al_*>0$, by Lemma \ref{lem1.1} it is readily seen that necessarily $\sg_1(\al_*, \psi_*)=0$. At this point, since $\al_*>0$ and in view of Proposition \ref{prop:simple},
     then, by Proposition \ref{prop:tranverse}
     we can continue once more the curve of solutions to a real analytic parametrization of the form $(-\eps,\eps)\ni s\mapsto (\lm(s),\al(s),\psi(s))$, where
    $(\lm(0),\al(0),\psi(0))=(\lm_*,\al_*,\psi_*)$ and without loss of generality we can assume that, for $s<0$, $(\lm(s),\al(s),\psi(s))$ coincides with
    the branch of unique solutions for $\lm<\lm_*$. As above in this situation we have $\lm^{'}(s)\neq 0$,
    $\sg_1(s)=\sg_1(\al,\psi)\left.\right|_{\lm=\lm(s),\al=\al(s),\psi=\psi(s)}\neq 0$ in $(-\eps,\eps)\setminus\{0\}$, $\sg_1(s)$ is simple in $(-\eps,\eps)$
    and \eqref{2907.5}, \eqref{2907.10} hold true, that is, for $\lm=\lm(s)$ we have,

    \begin{align}\label{2907.5.a}
                    \Abracket{[\phi_1]_{\lm_*},\psi_*}_{\lm_*}\neq 0\mbox{ and } \Abracket{[\phi_1]_{\lm_*},\psi_*}_{\lm_*} \mbox{   has the same sign as } \Abracket{\phi_1}_{\lm_*},
                \end{align}
                \begin{align}\label{2907.10.a}
                    \dfrac{\sg_1(s)}{\lm^{'}(s)}
                    =\dfrac{p\Abracket{[\phi_1]_{\lm_*},\psi_{*}}_{\lm_*}+\mbox{\rm o}(1)} {\Abracket{[\phi_1]_{\lm_*}^2}_{\lm_*}+\mbox{\rm o}(1)},\mbox{ as }s\to 0.
                \end{align}
    Passing to the limit as $s\to 0^-$ in \eqref{2907.10.a}, since $\lm^{'}(s)>0$ and $\sg_1(s)>0$, in view of \eqref{2907.5.a} we have that necessarily $<[\phi_1]_{\lm_*},\psi_{*}>_{\lm_*}>0$. Consequently we deduce again by \eqref{2907.10.a} that for $s\to 0^+$ small enough, $\lm^{'}(s)$ and $\sg_1(s)$ have
    the same sign. Thus, if $\sg_1(s)<0$ we also have $\lm^{'}(s)<0$ and the curve bends back to the region where $\lm<\lm_*$, which is of course a contradiction to the definition of $\lm_*$. Therefore necessarily $\sg_1(s)>0$ and $\lm^{'}(s)>0$, so that for $s>0$ small enough we would have $\al(s)>0$ and $\lm>\lm_*$
    and by the definition of $\lm_+$, the uniqueness for $\lm<\lm_*$ and Lemma \ref{lemE1} we infer the existence of a solution $(\al_-,\psi_-)$ for some $\al_-\leq 0$ for $\lm=\lm_+$, which is a contradiction to the uniqueness of solutions, see \cite{BandleSperb1983qualitative} for $N=2$ and Theorem C for $N\geq 3$.
\end{proof}

As an immediate consequence, in view of Proposition~\ref{prop:monotonicity of alpha-lambda I}, we have,

\begin{prop}\label{prop:monotonicity of alpha-lambda II}
    For any~$\lambda\in(0,+\infty)$, we have
    \begin{align}
        \frac{\dd\all}{\dd\lm}<0.
    \end{align}
\end{prop}


\section{Proof of Theorem~\ref{thm4}}\label{sec6}
\begin{proof}[Proof of Theorem~\ref{thm4}]
In view of Theorem \ref{thm:lm*} we have
$\lm_*(\mathbb{D}_{N},p)\geq \lm_+(\mathbb{D}_{N},p)$, implying by Theorem \ref{thm:gel} that
$\mathcal{G}_+(\om,p)$
is a real analytic branch and in particular that $\mathcal{R}_+$ as defined therein is a real analytic curve as well for $\lm<\lm_+(\mathbb{D}_N,p)$. The monotonicity of the energy is granted as well by Theorem \ref{thm:gel}, where we remark that,
as $\lm\nearrow\lm_+(\mathbb{D}_N,p)$, by the uniqueness of solutions $\pl$ converges in $C^2$
to a solution of \eqref{psiplus} above, implying that
$\el=\frac{1}{2}\ino \rl \pl\to E_\ii>E_0$. The behavior for $\lm$ small is well known and already described by Theorem A.\\

Therefore, at this point the crux of the proof of Theorem \ref{thm4} is just to show that the derivative of $\mul$ changes sign only once, which is
obtained by an elementary evaluation
about solutions of~$\pqm$. For later convenience we take the following shorthand notation,
\begin{align}
    \lm_+\equiv \lm_+(\mathbb{D}_N,p).
\end{align}

\noi\textbf{Case I: $N=2$.}
Let $\om=\mathbb{D}_{2}$.
Since $\mul=\lm\all^{p-1}$ then from \eqref{eq:change variable} we see that~$\gall^{p-1}=\mul$. Thus it is enough to prove the monotonicity assertion
for $\gall$. Set
\begin{align}
    r_{_\lm}\coloneqq \frac{R_2}{R(\lm)},
\end{align}
whence from Section~\ref{sect:calc in ball} it is readily seen that $r_{_\lm}$ is defined on $[0,\lm_+]$ and that $r_{_\bullet}:[0,\lm_+]\to [0,1]$ is monotonic increasing with $r_{_0}=0$, $r_{_{\lm_+}}=R_2=(\pi)^{-\frac12}$.
In particular by the definition of $R(\lm)$ in Section~\ref{sect:calc in ball} we have that,
\begin{align}\label{rlm}
R_2^{\frac{2}{p-1}}\lambda^{\frac{p}{p-1}}=r_{_\lm}^{\frac{2}{p-1}}{I_p(r_{_\lm})},
\end{align}
where
$$
I_p(r):=\int\limits_{B_r(0)}u_0^p,
$$
and~\eqref{eq:gamma-alpha-lambda-2} can be rewritten as follows,
\begin{align}\label{eq:gamma-lambda-2}
    R_2^{\frac{2}{p-1}}\gall=r_{_\lm}^{\frac{2}{p-1}}u_0(r_{_\lm}).
\end{align}
Taking the derivative of~\eqref{eq:gamma-lambda-2} with respect to~$\lm$, we have that,
\begin{align}
    R_2^{\frac{2}{p-1}} \frac{\dd \gall}{\dd{\lm}}
    = \left(\frac{2}{(p-1)}u_0(r_{_\lm})+r_{_\lm} u_0'(r_{_\lm})\right)r_{_\lm}^{\frac{2}{p-1}-1}\frac{\dd r_{_\lm} }{\dd \lm},
\end{align}
and our aim is to show that~$\frac{\dd \gall}{\dd\lm}$ changes sign only once.

Note that~$\frac{\dd r_{_\lm}}{\dd\lm}>0$ for~$\lm\in (0,\lm_+)$, meanwhile
\begin{align}
    g(\lm)\coloneqq \frac{2}{(p-1)}u_0(r_{_\lm})+r_{_\lm} u_0' (r_{_\lm}),
\end{align}
is continuous in $[0,\lm_+]$, strictly positive as $\lm\to 0^+$ (i.e. $r_{_\lm}\to 0^+$) and strictly negative as $\lm\to \lm_+^{-}$ (i.e. $r_{_\lm}\to R_2^-$).
Since~$u_0$ is radial, $u_0'<0$ and satisfies
\begin{align}
    u_0''+ \frac{1}{r}u_0'=- u_0^p
\end{align}
we see that
$$
g^{'}(\lm)=\left(\frac{2}{(p-1)}u_0'(r_{_\lm})+\frac{d}{dr}(r_{_\lm} u_0' (r_{_\lm}))\right)r^{'}_{_\lm}=
\left(\frac{2}{(p-1)}u_0'(r_{_\lm})-r_{_\lm} u_0^p(r_{_\lm}))\right)r^{'}_{_\lm}<0.
$$
Thus~$g$ is monotonic decreasing from positive to negative values, i.e.~$\gamma_\lambda$ (hence also~$\mul$) first increases and then decreases, turning down at some unique point~$\lm^t\in (0,\lm_+)$, as claimed.

\

\noi\textbf{Case II: $N\geq 3$. } The same idea applies to give the desired bending property in higher dimensions, only noting that some dimensional parameter has to be involved and~$p\in (1,p_{_N})$ for~$N\geq 3$.

Recall from Section~\ref{sect:calc in ball} (see also \cite{BJW2024sharp}) that if $\om=\mathbb{D}_{_N}$ and $N\geq 3$,
then $\prl$ (hence also~$\fbi$) admits a unique solution and that
$$
 \lm_+(\mathbb{D}_N,p)=I_+(\mathbb{D}_N,p)^{\frac{p-1}{p}}=\frac{I_p^{1-\frac{1}{p}}}{R_N^{\frac{N}{p}(1-\frac{p}{p_{_N}})}}.
$$

As before it suffices to prove the bending property for~$\mul$ since~$\mul=\gall^{p-1}$.
We will still use the convenient parameter~$r_{_\lm}=\frac{R_N}{R(\lm)}$.
It is readily seen that $r_{_\lm}$ is defined on $[0,\lm_+]$ (actually on~$[0,+\infty)$) and that $r_{_\bullet}:[0,\lm_+]\to [0,1]$ is monotonic increasing with $r_{_0}=0$,
$r_{_{\lm_+}}=R_N$.
In particular by the definition of $R(\lm)$ we have that,
\begin{align}\label{rlmN}
R_N^{\,\tau_{_N}}\lambda^{\frac{p}{p-1}}=r_{_\lm}^{\,\tau_{_N}}{I_p(r_{_\lm})}, \qquad
\mbox{ with } \tau_{_N}=\frac{N}{p-1}\parenthesis{1-\frac{p}{p_{_N}}}
\end{align}
where
$$
I_p(r):=\int\limits_{B_r(0)}u_0(y)^p\dd{y}.
$$
In view of~\eqref{eq:gamma-alpha-lambda-N},
$$
\gamma=\frac{1}{R(\lm)^{\frac{2}{p-1}}}u_0(\frac{R_{_N}}{R(\lm)}),
$$
we also have that,
\begin{align}\label{eq:gamma-lambda-3}
    R_N^{\frac{2}{p-1}}\gall=r_{_\lm}^{\frac{2}{p-1}}u_0(r_{_\lm}).
\end{align}
Thus, putting $u_0'\equiv \frac{d u_0}{dr}$ and taking the derivative of~\eqref{eq:gamma-lambda-3} with respect to~$\lm$, we need to prove that
$$
\left(\frac{2}{(p-1)}u_0(r_{_\lm})+r_{_\lm}u_0'(r_{_\lm})\right)r_{_\lm}^{\frac{2}{p-1}-1}r^{'}_{_\lm}
$$
changes sign only once.
In view of~$r^{'}_{_\lm}>0$, it is enough to consider the sign of the function,
\begin{align}
    g(\lm)\coloneqq\frac{2}{(p-1)}u_0(r_{_\lm})+r_{_\lm}u_0'(r_{_\lm}),
\end{align}
which is continuous in $[0,\lm_+]$, strictly positive as $\lm\to 0^+$ (i.e. $r_{_\lm}\to 0^+$) and strictly negative as $\lm\to \lm_+^{-}$ (i.e. $r_{_\lm}\to R_N^-$).
The equation for~$u_0$ in higher dimension becomes
\begin{align}
    u_0''+\frac{N-1}{r}u_0' = - u_0^p.
\end{align}
Hence we have
\begin{align}
    g^{'}(\lm)
    =& \left(\frac{2}{(p-1)}u_0'(r_{_\lm})+u_0'(r_{_\lm})+r_{_\lm}u_{0}''(r_{_\lm})\right)r^{'}_{_\lm} \\
    =& \left(\frac{2}{(p-1)}u_0'(r_{_\lm})-r_{_\lm}u_0^p-(N-2)u_0'(r_{_\lm})\right)r^{'}_{_\lm} \\
    =& \left( \frac{2}{(p-1)}\left(1-\frac{p-1}{p_{_N}-1}\right)u_0'(r_{_\lm})-r_{_\lm}u_0^p\right)r^{'}_{_\lm}<0,
\end{align}
and the conclusion follows, in this case as well.
\end{proof}

\bibliographystyle{siam}
\bibliography{Plasma}

\end{document}